\documentclass[onefignum,onetabnum]{siamart190516}

\usepackage[amsmath,thmmarks,hyperref]{ntheorem}

\NewDocumentCommand{\dgal}{sO{}m}{%
  \IfBooleanTF{#1}
    {\dgalext{#3}}
    {\dgalx[#2]{#3}}%
}

\NewDocumentCommand{\dgalext}{m}{%
  \sbox0{%
    \mathsurround=0pt 
    $\left\{\vphantom{#1}\right.\kern-\nulldelimiterspace$%
  }%
  \sbox2{\{}%
  \ifdim\ht0=\ht2
    \{\kern-.45\wd2 \{#1\}\kern-.45\wd2 \}%
  \else
  \fi
}

\NewDocumentCommand{\dgalx}{om}{%
  \sbox0{\mathsurround=0pt$#1\{$}%
  \sbox2{\{}%
  \ifdim\ht0=\ht2
    \{\kern-.45\wd2 \{#2\}\kern-.45\wd2 \}%
  \else
    \mathopen{#1\{\kern-.5\wd0 #1\{}
    #2
    \mathclose{#1\}\kern-.5\wd0 #1\}}
  \fi
}

\usepackage{amsfonts}
\usepackage{subfigure}
\usepackage{graphicx}
\usepackage{comment}
\usepackage{color}
\usepackage{mathrsfs}
\usepackage{multirow}
\usepackage{mathtools}
\usepackage{stmaryrd}
\usepackage{amsmath}
\usepackage{amssymb}
\usepackage{natbib}
\usepackage{bbm}      
\usepackage{mathrsfs}
\usepackage{graphicx}
\usepackage{algorithm}
\usepackage{algorithmic}
\usepackage{booktabs}
\usepackage{pgfplots}
\usepackage{mathrsfs}
\usetikzlibrary{arrows}
\usepackage{enumerate}
\usepackage{ulem}

\makeatletter
\def\tank#1{\protected@xdef\@thanks{\@thanks
		\protect\footnotetext[0]{#1}}}
\def\bigfoot{
	
	\@footnotetext}
\makeatother

\renewcommand{\theequation}{\arabic{section}.\arabic{equation}}
\newcommand{\ea}{\end{array}}
\newtheorem{example}{Example}[section]

{\theorembodyfont{\rmfamily}
}

\def\De{{\Delta}}

\def\De{{\Delta}}

\renewcommand{\theequation}{\arabic{section}.\arabic{equation}}

\pgfplotsset{compat=1.18}

\renewcommand{\theequation}{\arabic{section}.\arabic{equation}}

\headers{Bingru Huang, Falai Chen}{Bingru Huang, Falai Chen}

\title{Basis Construction for Spline Spaces over Arbitrary Partitions from a Dimensional Stable Perspective}

\author{Bingru Huang\thanks{Corresponding author. School of Mathematical Sciences, University of Science and Technology of China, Hefei 230026, China (\email{hbr999@ustc.edu.cn}).}
}

\ifpdf
\hypersetup{
  pdftitle={Arbitrary Partition},
  pdfauthor={Bingru Huang}
}
\fi

\begin{document}  

\maketitle

\begin{abstract}
This paper introduces a novel framework for constructing $C^r$ basis functions for polynomial spline spaces of degree $d$ over arbitrary planar polygonal partitions, overturning the belief that basis functions cannot be constructed on dimensional unstable meshes. We provide a comprehensive comparison of basis construction methods, classifying them as explicit, semi-implicit, and implicit. Our method, a semi-implicit construction using Extended Edge Elimination conditions, uniquely resolves fundamental theoretical challenges in the basis construction of spline spaces by ensuring a complete basis. For the first time, we construct explicit basis functions for the spline space over the Morgan-Scott partition, previously unachieved, and elucidate dimensional instability through this construction.
\end{abstract}




\section{Introduction}
Polynomial spline spaces defined over planar polygonal partitions, such as T-meshes and triangulations, have been a subject of study since the 1970s \cite{schumaker1979dimension, nurnberger2000developments}. These spaces are significant in approximation theory, numerical analysis \cite{mingjunlai2007, schumaker2007spline}, and isogeometric analysis \cite{cottrell2009isogeometric}, particularly in geometric modeling \cite{cohen2001geometric} and engineering applications \cite{cottrell2009isogeometric}. Research on polynomial spline spaces typically focuses on two fundamental problems: dimension calculation and basis construction.

For dimension calculation, spline spaces over triangulations have been extensively studied, with key results from the 1970s to 2000s summarized in Section 3 of \cite{nurnberger2000developments} and recent advancements in Section 2 of \cite{grovselj2023extraction, lyche2025c1}. For T-meshes, significant progress has been made over the past two decades \cite{dim2006, dim20061, dim2014, dim2016}. However, a critical challenge in both triangulation and T-mesh spline spaces is dimensional instability \cite{morgan1977dimension, alfeld1986dimension, gmelig1985dimension, shi1991singularity, deng2000note, Ins2011, Ins2012, huang2023,huang2025preliminarystudydimensionalstability}. This instability implies that the dimension of highly smooth spline spaces depends not only on topological properties but also on geometric configurations of the polygonal partition. A notable example is the Morgan-Scott triangulation \cite{morgan1977dimension}.

In basis construction, dimensional instability in highly smooth spline spaces over specific partitions complicates the direct construction of basis functions. Traditionally, researchers have sought to avoid issues arising from dimensional instability when constructing basis functions. For spline spaces over general polygonal partitions, basis functions have been constructed for cross-cut and quasi-cross-cut partitions \cite{wang2013multivariate, manni1992dimension}. Recent advancements for triangulations include new constructions \cite{speleers2013construction,lyche2022parsimonious, grovselj2023extraction, grovselj2017construction,lyche2025c1}. For T-meshes, notable constructions include T-splines \cite{ts1, ts2, ast}, HB-splines \cite{HB1, HB2}, LR-splines \cite{lr}, and PHT-splines \cite{PHT}. However, except for PHT-splines, these typically do not form a complete basis for the polynomial spline space over the corresponding T-mesh \cite{buffa2010linear, HB3, HB4, lr}. It must be emphasized that in non-basis function construction, the space spanned by linearly independent spline functions may not reproduce polynomials~\cite{ts1,ts2}, leading optimal convergence cannot be guaranteed.

This paper extends the framework of \citet{zhong2025basisconstructionpolynomialspline} to polynomial spline spaces over arbitrary planar polygonal partitions. We observe that any polygonal partition can be transformed into a cross-cut or quasi-cross-cut partition, which possess stable dimensions and well-defined basis function structures \cite{wang2013multivariate, manni1992dimension}. By applying the Extended Edge Elimination Condition (EEE condition), we derive a basis for the spline space over the original partition. To address the complexity of the EEE condition due to excessive extensions, we propose a simplified framework based on dimensional stability. This framework also facilitates the analysis of dimensional instability.

This study advances basis construction for spline spaces over arbitrary polygonal partitions, challenging the notion that basis functions cannot be constructed on dimensional unstable meshes. The key contributions are:

\begin{itemize}
    \item A generalized framework for constructing $C^r$ basis functions in polynomial spline spaces of arbitrary degree over any planar polygonal partition.
    \item A comprehensive comparison of mainstream basis construction methods, proposing three distinct approaches and highlighting the advantages of the proposed method.
    \item An analysis of basis construction under dimensional instability, using the Morgan-Scott partition to demonstrate that basis functions can be constructed for dimensional unstable polygonal partitions.
\end{itemize}

The paper is organized as follows. Section 2 reviews key concepts of spline spaces over polygonal partitions and foundational results on basis construction for cross-cut and quasi-cross-cut partitions. Section 3 presents the framework for constructing $C^r$ basis functions for polynomial spline spaces of arbitrary degree. Section 4 provides an example of basis construction for the Morgan-Scott partition, comparing it with existing methods, and proposes an optimized framework for practical applications. Section 5 concludes with a summary and directions for future work.

\section{Preliminaries}\label{sec.prelim}
In this section, we briefly introduce polynomial spline spaces over planar polygonal partitions, including triangulations and T-meshes. We also discuss cross-cut partitions, quasi-cross-cut partitions, and the construction of basis functions for these partition types.

Let $\Delta$ be a polygonal partition of a simply connected polygonal domain $\Omega$ in $\mathbb{R}^2$, we call such polygonal partition a \textbf{regular} polygonal partition.

For given integers $\mu, d, 0\le\mu\le d$, the spline space of degree $d$ with smoothness $\mu$ with respect to $\Delta$ is defined by
$$S_d^{\mu}(\Delta):=\{s(x,y)\in C^{\mu}(\Omega): s(x,y)|_{\phi}\in \Pi_{d} \ \ \text{for \ } \phi\in\Delta\}$$
where $\Pi_{d}:=\mathrm{span}\{x^iy^j:i,j\in\mathbb{N},0\le i+j\le d\}$ is the polynomial space of total degree $d$.

This paper constructs a basis for the spline space $S_d^{\mu}(\Delta)$ over an arbitrary polygonal partition $\Delta$. We consider two specific cases: triangulation and T-mesh. A triangulation of a domain $\Omega$ is a partition into a finite set of non-overlapping triangles, covering $\Omega$, with pairwise intersections only at shared vertices or edges. A T-mesh is a rectangular grid with a simply connected interior, permitting T-junctions. The spline spaces over these partitions are defined below.

\begin{itemize}
    \item \emph{Spline space over triangulations:} For a triangulation $\Delta$ of a polygonal domain, the spline space is defined as
    $$S_d^{\mu}(\Delta):=\{s(x,y)\in C^{\mu}(\Omega): s(x,y)|_{\phi}\in \Pi_{d} \ \ \text{for \ } \phi\in\Delta\}$$
    is called a spline space over triangulation $\Delta$.

    \item \emph{Spline space over T-mesh:} For a T-mesh $\Delta$ of a polygonal domain, the spline space over T-mesh is defined as
    $$S_{\textbf{d}}^{\boldsymbol{\mu}}(\Delta):=\{s(x,y)\in C^{\boldsymbol{\mu}}(\Omega): s(x,y)|_{\phi}\in \Pi_{d}\otimes\Pi_{d} \ \ \text{for \ } \phi\in\Delta\}$$
    where $\boldsymbol{\mu}=(\mu,\mu)$, $C^{\boldsymbol{\mu}}(\Omega)$ is the space consisting of all bivariate functions continuous in $\Omega$ with order $\mu$ along both the $x$-direction and $y$-direction.
\end{itemize}

The polynomial and smoothness spaces in spline spaces over T-meshes differ from those over polygonal partitions. For regular partitions, they are related as follows.
\begin{itemize}
\item For the polynomial space $\Pi_d \otimes \Pi_d$, we have $\Pi_d \otimes \Pi_d \subseteq \Pi_{2d}$, where $\subseteq$ denotes set inclusion, not subspace inclusion.
\item For the smoothness space $C^{\boldsymbol{\mu}}(\Omega)$, Lemma 3.1 in \cite{zeng2016dimensions} implies $S_{\boldsymbol{d}}^{\boldsymbol{\mu}}(\Delta) \subseteq C^{\mu}(\Omega)$.
\end{itemize}
Thus, for a regular T-mesh, the spline space is a subset of $S_{2d}^{\mu}(\Omega)$. Many results for spline spaces over polygonal partitions extend to those over T-meshes, though the latter possess additional properties due to their special definition.

We now present results on cross-cut and quasi-cross-cut partitions, including dimension formulas and basis constructions; details appear in \cite{wang2013multivariate,chui1983multivariate,chui1983smooth,manni1992dimension}.

For a polygonal partition of domain $D$, a straight line segment is a cross-cut if both endpoints lie on $\partial D$. A cross-cut partition, denoted $\Delta_c$, occurs when all grid lines are cross-cuts. A straight line segment is a ray if one endpoint lies on $\partial D$. A quasi-cross-cut partition, denoted $\Delta_{qc}$, occurs when each grid line is either a cross-cut or a ray. The dimensions of spline spaces over $\Delta_c$ and $\Delta_{qc}$ are given in \cite{chui1983multivariate,manni1992dimension}.

\begin{lemma}~\cite{chui1983multivariate}\label{lem dim cross-cut}
Let $\Delta_c$ be a cross-cut partition of $\Omega$ with $L$ cross-cuts and $V$ interior vertices $A_1,\ldots,A_V$ in $\Omega$ such that $N_i$ cross-cuts intersect at $A_i$ for $i=1,2,\ldots,V$. Then the dimension of spline space over $\Delta_c$ is 
$$\dim S_{d}^{\mu}(\Delta_c)=\binom{d+2}{2}+L\binom{d-\mu-1}{2}+\sum\limits_{i=1}^{V}k_{d}^{\mu}(N_i).$$
where $k_d^{\mu}(N_i)$ is defined as follow
\begin{equation}\label{kdmu}
    k_d^{\mu}(N_i)=\sum\limits_{j=1}^{d-\mu}\left(N_i(d-\mu-j+1)-(d-j+2)\right)_{+}.
\end{equation}
\end{lemma}

\begin{lemma}~\cite{chui1983multivariate,manni1992dimension}\label{lem dim quasi-cross-cut}
    Let $\Delta_{qc}$ be a quasi-cross-cut partition of $\Omega$ with $L_1$ cross-cuts, $L_2$ rays and $V$ interior vertices $A_1,\ldots,A_V$ in $\Omega$ such that $N_i$ cross-cuts and rays intersect at $A_i$ for $i=1,2,\ldots,V$. Then the dimension of spline space over $\Delta_{qc}$ is 
    $$\dim S_{d}^{\mu}(\Delta_{qc})=\binom{d+2}{2}+L_1\binom{d-\mu-1}{2}+\sum\limits_{i=1}^{V}k_{d}^{\mu}(N_i).$$
    where $k_d^{\mu}(\cdot)$ has the same meaning as Lemma~\ref{lem dim cross-cut}.
\end{lemma}

Lemmas~\ref{lem dim cross-cut} and~\ref{lem dim quasi-cross-cut} show that the dimension of the spline space over these partitions is stable, depending only on topological features (e.g., $L$, $N_i$). Bases for both partitions can be constructed using the following lemma.

\begin{lemma}~\cite{wang2013multivariate}\label{lem continuous condition}
    For any polygonal partition $\Delta$,let $\phi_i,\phi_j$ are two adjcent cells of $\Delta$, $s(x,y)|_{\phi_i}=s_i(x,y), s(x,y)|_{\phi_j}=s_j(x,y)$, and the equation of the common edge of $\phi_i$ and $\phi_j$ is $ L_{ij}(x,y)=0$. Then 
    $$s(x,y)\in C^{\mu}(\overline{\phi_i\cup\phi_j})\iff  s_i(x,y)-s_j(x,y)=q_{ij}(x,y)\cdot L_{ij}(x,y)^{\mu+1}.$$
    where $q_{i,j}(x,y)\in\Pi_{d-\mu-1}$ is called \textbf{the edge cofactor}.
\end{lemma}

Lemma~\ref{lem continuous condition} provides algebraic conditions for continuity across adjacent cells. By specifying flow directions between all adjacent cells and selecting a source cell, a basis for the spline space over cross-cut partition and quasi-cross-cut partition are obtained by truncating the power function. Details are provided in \cite{wang2013multivariate}. Due to the complexity of the description, the construction is deferred to later sections and clarified through an example.


\section{Framework for Basis Function Construction}
In this section, we introduce the concept of a \textbf{Basis Construction-Suitable Partition} for spline spaces, enabling direct basis construction without edge extension or elimination. We apply the Extended Edge Elimination (EEE) conditions to this polygonal partition, establishing a framework for constructing basis functions for the spline space. To clarify the construction process, we first consider the one-dimensional case before addressing the two-dimensional case.

\subsection{Motivation: One-Dimensional Case}
In the one-dimensional case, the partition $\Delta$ corresponds to a knot vector with interior knots of multiplicity $d-\mu$ and boundary knots of multiplicity $d+1$. We define knot vector refinement as follows.

\begin{definition}~\cite{lyche2008spline}
A knot vector $\Delta_1$ is a refinement of $\Delta_2$ if every real number in $\Delta_2$ appears in $\Delta_1$.
\end{definition}

This leads to the following lemma.

\begin{lemma}\label{lem one-dimension}
If $\Delta_1$ is a refinement of $\Delta_2$, then:
\[
S_d^\mu(\Delta_2) \subseteq S_d^\mu(\Delta_1),
\]
where $\subseteq$ denotes a subspace relation.
\end{lemma}

The follow corollary establish immediately.

\begin{corollary}\label{cor one-dimension represent}
$\forall s(x) \in S_d^\mu(\Delta_2)$ can be expressed as:
\[
s(x) = \sum c_i B_{i,\Delta_1}(x),
\]
where $B_{i,\Delta_1}(x)$ are B-splines defined by $\Delta_1$.
\end{corollary}

This study aims to construct a basis for the spline space $S_d^{\mu}(\Delta_2)$ from $S_d^{\mu}(\Delta_1)$, addressing the following problem:
\begin{equation}
\textbf{Construct a basis for } S_d^{\mu}(\Delta_2) \textbf{ from } S_d^{\mu}(\Delta_1).
\end{equation}

A key observation is that the primary distinction between $S_d^{\mu}(\Delta_1)$ and $S_d^{\mu}(\Delta_2)$ lies in the smoothness condition: in $S_d^{\mu}(\Delta_2)$, the knots in $\Delta_1 \setminus \Delta_2$ are $C^\infty$. This leads to the following lemma for $S_d^{\mu}(\Delta_2)$.

\begin{lemma}\label{lem one-dimension C^inf}
For $s(x) \in S_d^\mu(\Delta_2)$, being $C^\infty$ at knots in $\Delta_1 \setminus \Delta_2$ is equivalent to being $C^d$ at those knots.
\end{lemma}

An equivalent formulation of Lemma~\ref{lem one-dimension C^inf} is given below.

\begin{lemma}\label{lem one-dimension C^d}
For $s(x) \in S_d^\mu(\Delta_2)$, being $C^d$ at knots in $\Delta_1 \setminus \Delta_2$ is equivalent to
\begin{equation}
\left. \frac{\partial^r_{+}}{\partial x^r} s(x) \right|_{x \in \Delta_1 \setminus \Delta_2} = \left. \frac{\partial^r_{-}}{\partial x^r} s(x) \right|_{x \in \Delta_1 \setminus \Delta_2}\quad r=\mu+1,\ldots,d.
\end{equation}
\end{lemma}

Combining Lemma~\ref{lem one-dimension C^d} with Corollary~\ref{cor one-dimension represent}, we derive the following one-dimensional EEE condition.

\begin{lemma}
For $s(x) \in S_d^\mu(\Delta_2)$, being $C^d$ at knots in $\Delta_1 \setminus \Delta_2$ is equivalent to
\begin{equation}\label{eq one-dimension EEE}
\left. \sum c_i \left( \frac{\partial^r_{+}}{\partial x^r} B_{i,\Delta_1}(x) - \frac{\partial^r_{-}}{\partial x^r} B_{i,\Delta_1}(x) \right) \right|_{x \in \Delta_1 \setminus \Delta_2} = 0 \quad r=\mu+1,\ldots,d,
\end{equation}
which forms a linear system for the coefficients $c_i$.
\end{lemma}

The coefficients $c_i$ satisfying equation~\eqref{eq one-dimension EEE} ensure that $\sum c_i B_{i,\Delta_1}(x)$ is $C^d$ at $\Delta_1 \setminus \Delta_2$, and thus $C^\infty$. Consequently, $\sum c_i B_{i,\Delta_1}(x) \in S_d^\mu(\Delta_2)$, providing a basis for $S_d^\mu(\Delta_2)$.

\begin{theorem}
Let the linear system~\eqref{eq one-dimension EEE} have $h$ basis solutions $\boldsymbol{c}_j = (c_{1,j}, \ldots, c_{n,j})$ for $j=1,2,\ldots,h$, where $n = \dim S_d^\mu(\Delta_1)$. Then the set
\begin{equation}
\left\{ \sum c_{i,j} B_{i,\Delta_1}(x) \right\}_{j=1}^h
\end{equation}
forms a basis for $S_d^\mu(\Delta_2)$.
\end{theorem}

This provides the foundation for constructing a basis for the refined spline space $S_d^{\mu}(\Delta_2)$ from the coarser spline space $S_d^{\mu}(\Delta_1)$.

\subsection{Two-Dimensional Case}

In \cite{zhong2025basisconstructionpolynomialspline}, a basis construction method for spline spaces over diagonalizable T-meshes is introduced, followed by the use of EEE conditions to construct bases for spline spaces over arbitrary T-meshes. This section generalizes the basis construction for diagonalizable T-meshes, introducing the concept of Basis Construction-Suitable Partition. We then clarify the relationship and distinctions between Dimensional-Stable Partitions and Basis Construction-Suitable Partitions. Finally, we extend the EEE conditions to arbitrary polygonal partitions, enabling basis construction for spline spaces over such partitions. Before introducing the main result, we examine the extended polygonal partition, denoted $\mathrm{ext}_s(\Delta)$, as follows.

\begin{definition}
    For a polygonal partition $\Delta$, the extended partition $\mathrm{ext}_s(\Delta)$ is formed by extending the edges of $\Delta$, satisfying $\Delta \subseteq \mathrm{ext}_s(\Delta)$, where $s$ denotes the number of edges extended from $\Delta$ to $\mathrm{ext}_s(\Delta)$.
\end{definition}

For any polygonal partition $\Delta$, the maximal extension forms a cross-cut partition, termed \textbf{the extended cross-cut partition of $\Delta$}. The following lemma obviously holds.

\begin{lemma}\label{lem ascending chain}
    For the ascending chain:
    $$\mathrm{ext}_0(\De):=\De\subseteq \mathrm{ext}_1(\De)\subseteq \mathrm{ext}_2(\De)\subseteq\cdots\subseteq\mathrm{ext}_{s}(\De):=\De_c.$$
    where $\De_c$ is the extended cross-cut partition of $\De$. Then we have
    $$ S_{d}^{\mu}\left(\mathrm{ext}_0(\De)\right)\subseteq  S_{d}^{\mu}\left(\mathrm{ext}_1(\De)\right)\subseteq\cdots\subseteq  S_{d}^{\mu}\left(\mathrm{ext}_{s}(\De)\right).$$
    where '$\subseteq$' denotes a subspace relation.
\end{lemma}

We now define the Basis Construction-Suitable Partition (or Basis Construction-Suitable Mesh), which enables direct construction of basis functions for the spline space $S_d^{\mu}(\De)$ and serves as a core component in the basis construction framework.

\begin{definition}\label{def CSP}
    A polygonal partition $\Delta$ is termed a \emph{Basis Construction-Suitable Partition} if:
    \begin{itemize}
        \item The dimension of the spline space $S_d^{\mu}(\Delta)$ is stable.
        \item A basis for $S_d^{\mu}(\Delta)$ can be constructed \textbf{directly} on $\Delta$ \textbf{without extending its edges.}
    \end{itemize}
\end{definition}

Accordingly, we define the concept of \textbf{Dimensional-Stable Partition (or Dimensional-Stable Mesh)} as follow.

\begin{definition}\label{def DSP}
     A polygonal partition $\Delta$ is termed a Dimensional-Stable Partition if the dimension of the spline space $S_d^{\mu}(\Delta)$ is stable.
\end{definition}

By Definitions~\ref{def CSP} and~\ref{def DSP}, every Basis Construction-Suitable Partition is a Dimensional-Stable Partition, but the converse is not true. The Basis Construction-Suitable Partition depends on the basis construction method. A Dimensional-Stable Partition may form a Basis Construction-Suitable Partition under one construction method but not under another. Whatever, the dimensional stability condition is necessary.

\begin{example}~\cite{zhong2025basisconstructionpolynomialspline}
Consider the T-mesh $\De$ with a T-cycle, as shown in the left subfigure of Figure~\ref{fig:CSP & DSP}. For the spline space $S_{2,2}^{1,1}(\De)$, the dimension is stable. We construct basis functions for $S_{2,2}^{1,1}(\De)$ using two methods:
    \begin{itemize}
    \item \textbf{U-spline method}~\cite{thomas2022u}: Basis functions are constructed by forming linear systems from smoothness constraints on the Bézier coordinates of each cell and solving these systems. 
    
    For the T-mesh $\De$ depicted in the left subfigure of Figure~\ref{fig:CSP & DSP}, this method directly yields a basis for $S_{2,2}^{1,1}(\De)$, confirming that $\De$ is a Basis Construction-Suitable Partition.

    \item \textbf{PT-spline method}~\cite{zhong2025basisconstructionpolynomialspline}: Basis functions are constructed by extending edges in the T-mesh to form an extended T-mesh $\mathrm{ext}_s(\De)$ that supports local tensor product B-splines, satisfying the dimension formula for $S_{\boldsymbol{d}}^{\boldsymbol{\mu}}(\mathrm{ext}_s(\De))$. The Extended Edge Elimination Condition is then applied to obtain a basis for $S_{\boldsymbol{d}}^{\boldsymbol{\mu}}(\mathrm{ext}_s(\De))$. 
    
    For the T-mesh $\De$ depicted in the left subfigure of Figure~\ref{fig:CSP & DSP}, edge extensions are necessary to satisfy the PT-spline construction conditions, as depicted in the right subfigure of Figure~\ref{fig:CSP & DSP}. Consequently, the extended T-mesh $\mathrm{ext}_2(\De)$ forms a Basis Construction-Suitable Partition. See Example 11 in~\cite{zhong2025basisconstructionpolynomialspline} for details.
\end{itemize}
\end{example}
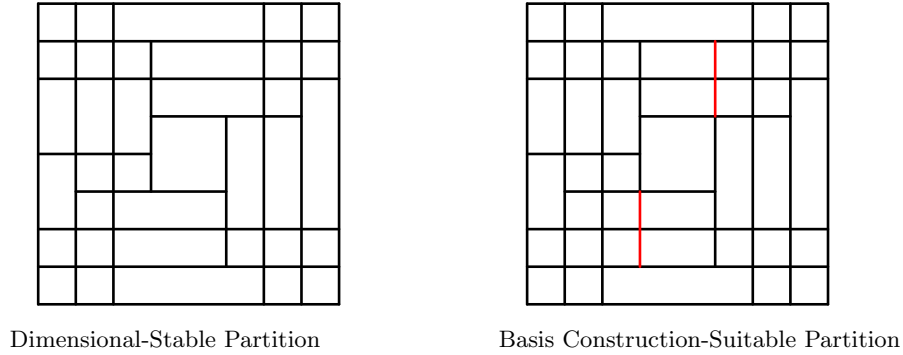
\begin{figure}
    \centering
    \begin{tikzpicture}[line cap=round,line join=round,>=triangle 45,x=0.5cm,y=.5cm,scale=1]
\clip(-4.66,-2.16) rectangle (25.16,12.84);
\draw [line width=1.pt] (0.,1.)-- (8.,1.);
\draw [line width=1.pt] (0.,1.)-- (0.,9.);
\draw [line width=1.pt] (0.,9.)-- (8.,9.);
\draw [line width=1.pt] (8.,9.)-- (8.,1.);
\draw [line width=1.pt] (1.,1.)-- (1.,9.);
\draw [line width=1.pt] (2.,9.)-- (2.,1.);
\draw [line width=1.pt] (7.,9.)-- (7.,1.);
\draw [line width=1.pt] (6.,1.)-- (6.,9.);
\draw [line width=1.pt] (0.,8.)-- (8.,8.);
\draw [line width=1.pt] (0.,7.)-- (8.,7.);
\draw [line width=1.pt] (3.,8.)-- (3.,4.);
\draw [line width=1.pt] (1.,4.)-- (5.,4.);
\draw [line width=1.pt] (0.,2.)-- (8.,2.);
\draw [line width=1.pt] (8.,3.)-- (0.,3.);
\draw [line width=1.pt] (5.,2.)-- (5.,6.);
\draw [line width=1.pt] (3.,6.)-- (7.,6.);
\draw [line width=1.pt] (3.,5.)-- (0.,5.);
\draw [line width=1.pt] (13.,9.)-- (13.,1.);
\draw [line width=1.pt] (13.,1.)-- (21.,1.);
\draw [line width=1.pt] (21.,1.)-- (21.,9.);
\draw [line width=1.pt] (21.,9.)-- (13.,9.);
\draw [line width=1.pt] (13.,8.)-- (21.,8.);
\draw [line width=1.pt] (13.,7.)-- (21.,7.);
\draw [line width=1.pt] (14.,9.)-- (14.,1.);
\draw [line width=1.pt] (15.,1.)-- (15.,9.);
\draw [line width=1.pt] (13.,2.)-- (21.,2.);
\draw [line width=1.pt] (13.,3.)-- (21.,3.);
\draw [line width=1.pt] (20.,1.)-- (20.,9.);
\draw [line width=1.pt] (19.,9.)-- (19.,1.);
\draw [line width=1.pt] (16.,8.)-- (16.,4.);
\draw [line width=1.pt] (14.,4.)-- (18.,4.);
\draw [line width=1.pt] (18.,2.)-- (18.,6.);
\draw [line width=1.pt] (20.,6.)-- (16.,6.);
\draw [line width=1.pt] (16.,5.)-- (13.,5.);
\draw [line width=1.pt,color=red] (16.,4.)-- (16.,2.);
\draw [line width=1.pt,color=red] (18.,6.)-- (18.,8.);
\draw (-1,0.54) node[anchor=north west] {\small Dimensional-Stable\ Partition};
\draw (12,0.54) node[anchor=north west] {\small Basis Construction-Suitable\ Partition};
\end{tikzpicture}
    \caption{\label{fig:CSP & DSP}Comparison of Dimensional-Stable Partition and Basis Construction-Suitable Partition}
\end{figure}

\begin{definition}
For an arbitrary polygonal partition $\Delta$, let $\mathcal{M}$ denote a basis construction method. The pair $(\Delta, \mathcal{M})$ is \textbf{compatible} if $\Delta$ is a Basis Construction-Suitable partition under $\mathcal{M}$.
\end{definition}

Many partitions studied for constructing spline basis functions are Basis Construction-Suitable Partitions, enabling direct basis construction. In the T-mesh category, these include hierarchical meshes with mild restrictions for HB-splines~\cite{HB3}, AS T-meshes for AS T-splines~\cite{ast,li2014analysis}, restricted LR-meshes for LR-splines~\cite{patrizi2020adaptive}, and hierarchical T-meshes for PHT-splines~\cite{PHT}. In the triangulation category, examples include Clough-Tocher (CT) splits~\cite{mingjunlai2007,sablonniere1987composite} for $S_d^2(\Delta_{\mathrm{CT}})$ with $d \geq 7$ and Powell-Sabin (PS) 6 and 12 splits~\cite{powell1977piecewise,alfeld2002smooth,schumaker2006smooth,mingjunlai2007,speleers2013construction} for $S_d^2(\Delta_{\mathrm{PS}})$ with $d \geq 5$ etc.

\begin{example}
Consider a hierarchical T-mesh $\Delta$ shown in Figure~\ref{fig:HB PT PHT} and the basis construction methods for HB-spline~\cite{HB2} and PHT-spline~\cite{PHT}, denoted $\mathcal{M}_{\text{HB}}$ and $\mathcal{M}_{\text{PHT}}$, respectively. The compatibility of these methods with $\Delta$ for different spline spaces is summarized in Table~\ref{tab:compatibility}.

\begin{table}[ht]
\centering
\caption{Compatibility of construction methods with $\Delta$}
\label{tab:compatibility}
\begin{tabular}{|c|c|c|}
\hline
Spline Space & $\mathcal{M}_{\text{HB}}$ & $\mathcal{M}_{\text{PHT}}$ \\
\hline
$S_2^1(\Delta)$ & Not compatible & Not compatible \\
$S_3^1(\Delta)$ & Not compatible & Compatible \\
\hline
\end{tabular}
\end{table}

The pair $(\Delta, \mathcal{M}_{\text{HB}})$ is not compatible because $\Delta$ does not form a completeness hierarchical B-spline mesh~\cite{HB3}. For $S_2^1(\Delta)$, $(\Delta, \mathcal{M}_{\text{PHT}})$ is not compatible as the smoothness and degree do not satisfy PHT-spline requirements, which mandate smoothness less than half the degree. For $S_3^1(\Delta)$, $(\Delta, \mathcal{M}_{\text{PHT}})$ is compatible since $\Delta$ is a hierarchical T-mesh and the smoothness meets PHT-spline requirements.
\end{example}

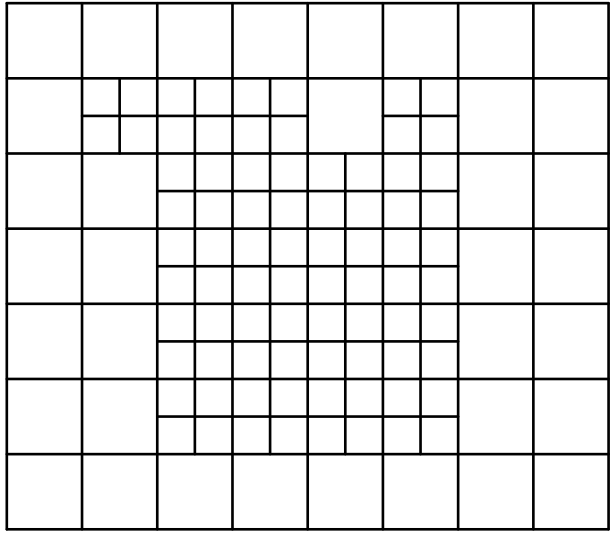
\begin{figure}
    \centering
    \begin{tikzpicture}[line cap=round,line join=round,>=triangle 45,x=1.0cm,y=1.0cm]
\draw [line width=1.pt] (0.,0.)-- (8.,0.);
\draw [line width=1.pt] (8.,0.)-- (8.,7.);
\draw [line width=1.pt] (8.,7.)-- (0.,7.);
\draw [line width=1.pt] (0.,7.)-- (0.,0.);
\draw [line width=1.pt] (0.,6.)-- (8.,6.);
\draw [line width=1.pt] (0.,5.)-- (8.,5.);
\draw [line width=1.pt] (0.,4.)-- (8.,4.);
\draw [line width=1.pt] (0.,3.)-- (8.,3.);
\draw [line width=1.pt] (0.,2.)-- (8.,2.);
\draw [line width=1.pt] (0.,1.)-- (8.,1.);
\draw [line width=1.pt] (1.,0.)-- (1.,7.);
\draw [line width=1.pt] (2.,7.)-- (2.,0.);
\draw [line width=1.pt] (3.,7.)-- (3.,0.);
\draw [line width=1.pt] (4.,0.)-- (4.,7.);
\draw [line width=1.pt] (5.,7.)-- (5.,0.);
\draw [line width=1.pt] (6.,0.)-- (6.,7.);
\draw [line width=1.pt] (7.,7.)-- (7.,0.);
\draw [line width=1.pt] (1.,5.5)-- (4.,5.5);
\draw [line width=1.pt] (5.,5.5)-- (6.,5.5);
\draw [line width=1.pt] (2.,4.5)-- (6.,4.5);
\draw [line width=1.pt] (2.,3.5)-- (6.,3.5);
\draw [line width=1.pt] (2.,2.5)-- (6.,2.5);
\draw [line width=1.pt] (2.,1.5)-- (6.,1.5);
\draw [line width=1.pt] (2.5,6.)-- (2.5,1.);
\draw [line width=1.pt] (3.5,1.)-- (3.5,6.);
\draw [line width=1.pt] (4.5,5.)-- (4.5,1.);
\draw [line width=1.pt] (5.5,1.)-- (5.5,6.);
\draw [line width=1.pt] (1.5,6.)-- (1.5,5.);
\end{tikzpicture}
    \caption{\label{fig:HB PT PHT}A hierarchical T-mesh}
\end{figure}

In fact, cross-cut and quasi-cross-cut partitions are both Basis Construction-Suitable Partitions for the construction method mentioned in~\cite{wang2013multivariate}. The cross-cut partition is the maximal Basis Construction-Suitable Partition for a given polygonal partition $\Delta$.

We now define the Eliminate Extended Edge Condition for an edge, which removes the smoothness constraint at that edge.

\begin{definition}\label{def EEE}
    Given a polygonal partition $\De$ and extended Basis Construction-Suitable Partition $\mathrm{ext}_s(\De)$, as defined in Definition~\ref{def CSP}, assume that $\{N_1, N_2, \ldots, N_n\}$ forms a basis for $S_d^{\mu}(\mathrm{ext}_s(\De))$. For an extended edge $l$ in $\mathrm{ext}_s(\De)$, let $\phi_1$ and $\phi_2$ be two adjacent cells intersecting at $l$. The condition
\begin{equation}\label{eq EEE}
    \sum_{i=1}^{n} c_i N_i \in C^d(\overline{\phi_1 \cup \phi_2})
\end{equation}
is termed the Extended Edge Elimination (EEE) condition for $l$.
\end{definition}

In $S_d^{\mu}(\mathrm{ext}_s(\De))$, the polynomial degree in each cell is $d$. Thus, the EEE condition is equivalent to the condition $\sum_{i=1}^{n} c_i N_i \in C^{\infty}(\overline{\phi_1 \cup \phi_2})$, ensures infinite differentiability at edge $l$. This property motivates the term Extended Edge Elimination (EEE) condition. We now present an equivalent algebraic formulation for the EEE condition at $l$.

\begin{lemma}
Let the conditions of Definition~\ref{def EEE} hold, and let the equation of line $l$ be $L=0$, where $L \in \Pi_1$ and $\boldsymbol{n}$ is a unit normal vector of $l$. The EEE condition for $l$ is equivalent to
\begin{equation}
\sum_{i=1}^n c_i \cdot \frac{\partial^r}{\partial \boldsymbol{n}^r} \left( N_i|_{\phi_1} - N_i|_{\phi_2} \right) \Big|_l \equiv 0, \quad r = \mu+1, \ldots, d,
\end{equation}
where $\frac{\partial^r}{\partial \boldsymbol{n}^r}$ denotes the $r$-th order directional derivative along $\boldsymbol{n}$.
\end{lemma}

\begin{proof}
Since $N_i \in S_d^\mu(\mathrm{ext}_s(\Delta))$, the EEE condition~\eqref{eq EEE} implies
\[
\sum_{i=1}^n c_i N_i \in C^d(\overline{\phi_1 \cup \phi_2}) \iff \sum_{i=1}^n c_i N_i \in C^r(\overline{\phi_1 \cup \phi_2}), \quad r = \mu+1, \ldots, d.
\]
Given $\sum_{i=1}^n c_i N_i \in C^0(\overline{\phi_1 \cup \phi_2})$, it follows that
\[
\sum_{i=1}^n c_i \left( N_i|_{\phi_1} - N_i|_{\phi_2} \right) \Big|_l \equiv 0.
\]
Thus, all directional derivatives of $\sum_{i=1}^n c_i (N_i|_{\phi_1} - N_i|_{\phi_2})$ along $l$ vanish, yielding
\[
\sum_{i=1}^n c_i N_i \in C^r(\overline{\phi_1 \cup \phi_2}) \iff \sum_{i=1}^n c_i \cdot \frac{\partial^r}{\partial \boldsymbol{n}^r} \left( N_i|_{\phi_1} - N_i|_{\phi_2} \right) \Big|_l \equiv 0, \quad r = \mu+1, \ldots, d.
\] 

$\square$

\end{proof}

The EEE condition for $l$ forms a linear system $M_l \boldsymbol{c} = 0$, where $\boldsymbol{c} = (c_1, c_2, \ldots, c_n)^\top$. The global EEE conditions for all extended edges are defined as follows.

\begin{definition}
Given a polygonal partition $\Delta$ and extended Basis Construction-Suitable partition $\mathrm{ext}_s(\Delta)$, let $\{N_1, N_2, \ldots, N_n\}$ be a basis for $S_d^\mu(\mathrm{ext}_s(\Delta))$. For the set of extended edges $\{l_1, l_2, \ldots, l_s\}$, the EEE conditions for each $l_i$, $i = 1, \ldots, s$, form the linear system
\begin{equation}\label{eq global EEE}
M \boldsymbol{c} = 0,
\end{equation}
where $M = \begin{pmatrix} M_{l_1} & M_{l_2} & \cdots & M_{l_s} \end{pmatrix}^\top$ and $\boldsymbol{c} = (c_1, c_2, \ldots, c_n)^\top$. This system is termed the \textbf{Global Extended Edge Elimination Conditions} (EEE conditions).
\end{definition}

\begin{theorem}\label{thm basis construction}
Given a polygonal partition $\Delta$ and its extended Basis Construction-Suitable partition $\mathrm{ext}_s(\Delta)$, let $\{N_1, N_2, \ldots, N_n\}$ be a basis for $S_d^\mu(\mathrm{ext}_s(\Delta))$. If the homogeneous linear system $M \boldsymbol{c} = 0$ from~\eqref{eq global EEE} has a basis of solutions $\{\tilde{c}_j\}_{j=1}^h$, where $\tilde{c}_j = (\tilde{c}_{1,j}, \tilde{c}_{2,j}, \ldots, \tilde{c}_{n,j})$, then the functions
\[
\left\{ \sum_{i=1}^n \tilde{c}_{i,j} N_i \right\}_{j=1}^h
\]
form a basis for $S_d^\mu(\Delta)$, and consequently, $h = \dim S_d^\mu(\Delta)$.
\end{theorem}

\begin{proof}
      Let \(\{l_1, l_2, \ldots, l_s\}\) denote the set of all extended edges. The proof proceeds in two steps:

\begin{enumerate}
    \item For any \(s \in S_{d}^{\mu}(\Delta)\), Lemma~\ref{lem ascending chain} implies \(s \in S_{d}^{\mu}(\mathrm{ext}_s(\Delta))\), allowing the representation
      \[
      s = \sum_{i=1}^n s_i N_i.
      \]
      Since \(s \in S_{d}^{\mu}(\Delta)\), for each virtual edge \(l_i\), \(i = 1, 2, \ldots, s\), it follows $\sum_{i=1}^{n} s_i N_i \in C^{\infty}\left(\overline{\phi_1^{{(i)}} \cup \phi_2^{{(i)}}}\right)$
      where $\phi_1^{{(i)}}, \phi_2^{{(i)}}$ are two adjacent cells of $l_i$. Consequently, in \(\Delta\), the following conditions hold:
           $$\sum_{i=1}^n s_i \cdot \frac{\partial^r}{\partial \boldsymbol{n}^r} \left( N_i|_{\phi_1^{{(i)}}} - N_i|_{\phi_2^{(i)}} \right) \Big|_{l_i} \equiv 0, \quad r = \mu+1, \ldots, d,$$
          where $\phi_1^{{(i)}}, \phi_2^{{(i)}}$ are two adjacent cells of $l_i$.

      Thus, \((s_1, s_2, \ldots, s_n)\) satisfies the EEE linear system \(M\boldsymbol{c} = 0\), yielding
      \[
      (s_1, s_2, \ldots, s_n) = \sum_{j=1}^h k_j \tilde{c}_j,
      \]
      where \(\tilde{c}_j = (\tilde{c}_{1,j}, \tilde{c}_{2,j}, \ldots, \tilde{c}_{n,j})\) and \(k_j\), \(j = 1, 2, \ldots, h\), are constants. Consequently,
      \[
      s = \sum_{j=1}^h k_j \left( \sum_{i=1}^n \tilde{c}_{i,j} N_i \right),
      \]
      implying \(s \in \mathrm{span} \left\{ \sum_{i=1}^n \tilde{c}_{i,j} N_i \right\}_{j=1}^h\). Hence, \(S_{d}^{\mu}(\Delta) \subseteq \mathrm{span} \left\{ \sum_{i=1}^n \tilde{c}_{i,j} N_i \right\}_{j=1}^h\).

      Conversely, under the EEE conditions, \(\sum_{i=1}^n \tilde{c}_{i,j} N_i \in S_{d}^{\mu}(\Delta)\) for \(j = 1, 2, \ldots, h\), so 
      \[\mathrm{span} \left\{ \sum_{i=1}^n \tilde{c}_{i,j}N_i \right\}_{j=1}^h \subseteq S_{d}^{\mu}(\Delta)\]. 
      
      Therefore,
      \[
      S_{d}^{\mu}(\Delta) = \mathrm{span} \left\{ \sum_{i=1}^n \tilde{c}_{i,j}N_i \right\}_{j=1}^h.
      \]

    \item We establish the linear independence of the functions in \(\left\{ \sum_{i=1}^n \tilde{c}_{i,j}N_i \right\}_{j=1}^h\).

      Assume \(\sum_{j=1}^h d_j \left( \sum_{i=1}^n \tilde{c}_{i,j} N_i \right) \equiv 0\). We prove \(d_j = 0\) for \(j = 1, 2, \ldots, h\).

      This equation simplifies to
      \[
      \sum_{i=1}^n \left( \sum_{j=1}^h d_j \tilde{c}_{i,j} \right) N_i = 0,
      \]
      and since \(\{N_1, N_2, \ldots, N_n\}\) is a basis for \(S_{d}^{\mu}(\mathrm{ext}_s(\Delta))\), it follows that
      \[
      \sum_{j=1}^h d_j \tilde{c}_{i,j} = 0, \quad i = 1, 2, \ldots, n.
      \]
      This can be expressed as
      \[
      (d_1, d_2, \ldots, d_h) \begin{pmatrix} \tilde{c}_1 \\ \tilde{c}_2 \\ \vdots \\ \tilde{c}_h \end{pmatrix} = 0.
      \]
      Given that \(\{\tilde{c}_j\}_{j=1}^h\) forms a linearly independent basis of solutions, \(d_j = 0\) for \(j = 1, 2, \ldots, h\).
\end{enumerate}

In conclusion, \(S_{d}^{\mu}(\Delta) = \mathrm{span} \left\{ \sum_{i=1}^n \tilde{c}_{i,j} N_i \right\}_{j=1}^h\), and the functions in \(\left\{ \sum_{i=1}^n \tilde{c}_{i,j} N_i \right\}_{j=1}^h\) are linearly independent. Thus, these functions form a basis for \(S_{d}^{\mu}(\Delta)\), and consequently, \(h = \dim S_{d}^{\mu}(\Delta)\).

$\square$

\end{proof}

The construction above, termed semi-implicit, combines basis functions over $\mathrm{ext}_s(\Delta)$ with the EEE conditions and is detailed in Section 4.

In the following, based on the concept of Basis Construction-Suitable Partitions and EEE conditions, we provide the framework for the construction of basis functions for the spline space over arbitrary polygonal partitions. Before presenting the framework, we establish a lemma that ensures the existence of a compatible pair $(\mathrm{ext}_s(\Delta), \mathcal{M})$.

\begin{lemma}\label{lem compatible}
For any polygonal partition $\Delta$, there exist an extended partition $\mathrm{ext}_s(\Delta)$ and a basis construction method $\mathcal{M}$ such that the pair $(\mathrm{ext}_s(\Delta), \mathcal{M})$ is compatible.
\end{lemma}

\begin{proof}
Select the cross-cut partition $\mathrm{ext}_s(\Delta)$ extended from $\Delta$ and adopt the basis construction method $\mathcal{M}$ for the spline space over cross-cut partitions, as described in~\cite{wang2013multivariate}. This ensures that, for any polygonal partition $\Delta$, a basis for $S_d^\mu(\mathrm{ext}_s(\Delta))$ can always be constructed.

$\square$

\end{proof}

The framework for constructing a basis for the polynomial spline space $S_d^\mu(\Delta)$ is as follows:
\begin{itemize}
    \item Given a polygonal partition $\Delta$, select a basis construction method $\mathcal{M}$.
    \item If $(\Delta, \mathcal{M})$ is compatible, proceed directly. Otherwise, extend $\Delta$ to $\mathrm{ext}_s(\Delta)$ such that $(\mathrm{ext}_s(\Delta), \mathcal{M})$ is compatible. If no such $\mathrm{ext}_s(\Delta)$ exists for $\mathcal{M}$, select a new method $\mathcal{M}'$ and an extended partition $\mathrm{ext}_{s'}(\Delta)$ such that $(\mathrm{ext}_{s'}(\Delta), \mathcal{M}')$ is compatible, as guaranteed by Lemma~\ref{lem compatible}. Thus, a compatible pair $(\mathrm{ext}_s(\Delta), \mathcal{M})$ is obtained.
    \item Construct a basis $\{N_1, N_2, \ldots, N_n\}$ for $S_d^\mu(\mathrm{ext}_s(\Delta))$ using $\mathcal{M}$.
    \item By Theorem~\ref{thm basis construction}, combine $\{N_1, N_2, \ldots, N_n\}$ with the EEE conditions to form a basis for $S_d^\mu(\Delta)$.
\end{itemize}

\begin{figure}
    \centering
\begin{tikzpicture}[
  box/.style={
    rectangle,
    draw,
    minimum height=2em,
    minimum width=4em, 
    text centered,
    font=\sffamily
  },
  arrow/.style={
    ->,
    thick
  }
]

\node[box] (A) at (0,0) {$\mathrm{Initial \ Partition:}\ \Delta$};
\node[box] (B) at (0,-3) {\text{CSP}:$\mathrm{ext}_0(\Delta)$};
\node[box] (C) at (6,-3) {\text{ CSP:}$\mathrm{ext}_s(\Delta)$};
\node[box] (D) at (6,-6) {$S_d^{\mu}(\mathrm{ext}_s(\Delta))$};
\node[box] (E) at (0,-6) {$S_d^{\mu}(\De)$};

\draw[arrow] (A) -- (B) node[midway, left] {No Extended Edges};
\draw[arrow] (A) -- (C) node[midway, right] {Extended Edges};
\draw[arrow] (B) -- (E) node[midway, left] {Directly Construction};
\draw[arrow] (C) -- (D) node[midway, right] {Directly Construction};
\draw[arrow] (D) -- (E) node[midway, below] {EEE Conditions};

\end{tikzpicture}
    \caption{\label{fig:framework}Framework for Constructing a Basis for the Polynomial Spline Space over Arbitrary Polygonal Partition}
\end{figure}
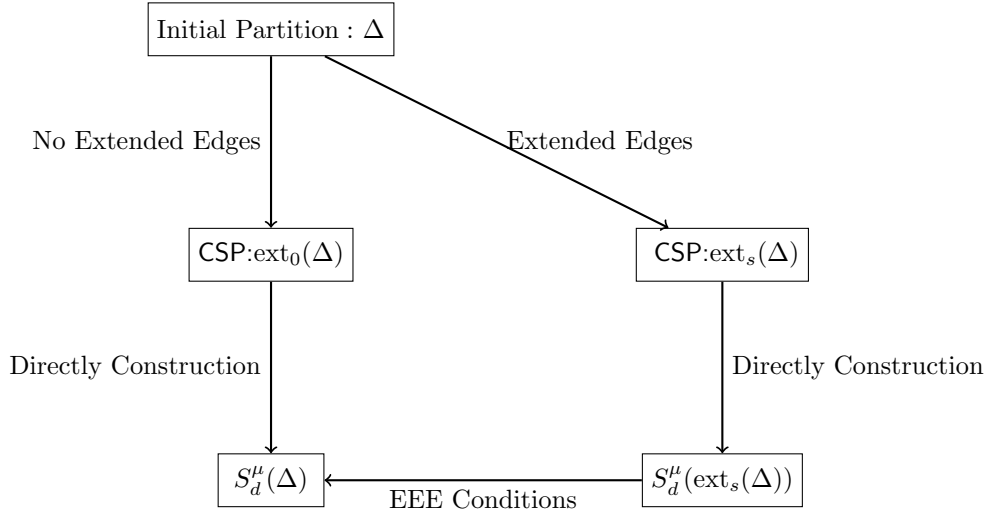

Figure~\ref{fig:framework} illustrates the basis construction process, where CSP denotes Basis Construction-Suitable Partition. The following Algorithm~\ref{alg:basis_construction} outlines the framework for constructing a basis for the polynomial spline space $S_d^\mu(\Delta)$.

\begin{algorithm}[htb]
\caption{Basis Construction for Polynomial Spline Space}
\label{alg:basis_construction}
\begin{algorithmic}[1]
\REQUIRE ~~\\
A polygonal partition $\Delta$, degree $d$, continuity $\mu$;
\ENSURE ~~\\
A basis for $S_d^\mu(\Delta)$;
\STATE Select a basis construction method $\mathcal{M}$;
\IF {$(\Delta, \mathcal{M})$ is compatible}
    \STATE Proceed with $\Delta$ and $\mathcal{M}$;
\ELSE
    \STATE Find an extended partition $\mathrm{ext}_s(\Delta)$ such that $(\mathrm{ext}_s(\Delta), \mathcal{M})$ is compatible;
    \IF {no such $\mathrm{ext}_s(\Delta)$ exists for $\mathcal{M}$}
        \STATE Select a new method $\mathcal{M}'$ and extended partition $\mathrm{ext}_{s'}(\Delta)$ such that $(\mathrm{ext}_{s'}(\Delta), \mathcal{M}')$ is compatible;
        \STATE Set $\mathcal{M} \gets \mathcal{M}'$, $\mathrm{ext}_s(\Delta) \gets \mathrm{ext}_{s'}(\Delta)$;
    \ENDIF
\ENDIF
\STATE Construct basis $\{N_1, N_2, \ldots, N_n\}$ for $S_d^\mu(\mathrm{ext}_s(\Delta))$ using $\mathcal{M}$;
\STATE Solve the homogeneous linear system $M \boldsymbol{c} = 0$ from the EEE conditions to obtain solution basis $\{\tilde{c}_j\}_{j=1}^h$;
\STATE Compute basis for $S_d^\mu(\Delta)$ as $\left\{ \sum_{i=1}^n \tilde{c}_{i,j} N_i \right\}_{j=1}^h$;
\STATE Output $\left\{ \sum_{i=1}^n \tilde{c}_{i,j} N_i \right\}_{j=1}^h$.
\end{algorithmic}
\end{algorithm}

Although Definition~\ref{def EEE} of the EEE condition requires that all points on an edge \( l \), including its endpoints, satisfy the condition, Algorithm~\ref{alg:basis_construction} constructs basis functions for the spline space \( S_d^{\mu}(\Delta) \), where the \( C^k \) smoothness condition is inherently satisfied at all interior vertices of \( \Delta \). Consequently, the EEE condition can be relaxed to apply only to the edges, excluding their interior vertices.

Finally, we highlight the role of the proposed framework in applications such as geometric modeling and analysis requiring adaptive refinement. The adaptive refinement process involves:
\begin{itemize}
    \item \textbf{Initial Setup}: Given a partition $\Delta$, an error threshold $\epsilon$, and an objective function $f$, identify cells in $\Delta$ where the error of $f$ exceeds $\epsilon$ for refinement.
    \item \textbf{Refinement}: Refine these cells to create a new partition $\Delta'$, and construct basis functions on $\Delta'$ for subsequent error assessment.
\end{itemize}

Constructing basis functions on $\Delta'$ often requires extending edges to form $\mathrm{ext}_s(\Delta')$, causing refinement of cells that already satisfy $\epsilon$. This increases computational cost and reduces efficiency. Our framework employs the EEE condition \cite{zhong2025basisconstructionpolynomialspline} to target only cells with errors above $\epsilon$, minimizing unnecessary refinements. This precision enhances efficiency in adaptive refinement, as refining cells solely for basis construction can amplify errors in later iterations, requiring further refinements.

For instance, HB-spline construction \cite{HB1} requires a hierarchical mesh for local tensor-product B-splines, leading to additional refinement of cells meeting threshold $\epsilon$. In contrast, PHT-splines \cite{PHT} avoid this by using cross-division refinement for cells not meeting error threshold $\epsilon$, eliminating the need for edge extensions.

\section{Comparisons and Examples}
In this section, we address two objectives: first, we compare the proposed basis construction framework with existing methods; second, we analyze dimensional instability in the Morgan-Scott partition by constructing a basis and examining degeneration through EEE conditions.

\subsection{Comparisons}
This subsection compares the proposed framework with other basis construction methods. First, we compare our method with the knot insertion and the Oslo algorithm to highlight its novelty. Second, using spline space over T-mesh as an example, we conduct a comprehensive comparison of basis construction methods, classifying them into three categories: explicit, implicit, and semi-implicit constructions. The proposed framework is semi-implicit. Additionally, we briefly compare spline basis construction with finite element basis construction to clarify their distinctions.

\subsubsection{A Comparison of One-Dimensional Case}
The knot insertion algorithm \cite{boehm1980inserting} and the Oslo algorithm \cite{cohen1980discrete} are established methods for B-spline basis construction, transforming bases from a refined spline space $S_d^\mu(\Delta_2)$ to a coarser spline space $S_d^\mu(\Delta_1)$. This subsection compares these algorithms with our proposed method.

By Lemma~\ref{lem one-dimension}, all B-splines in $S_d^\mu(\Delta_2)$ belong to $S_d^\mu(\Delta_1)$, expressed as:
\begin{equation}
B_{\Delta_2} = B_{\Delta_1} A,
\end{equation}
where $B_{\Delta_i}$ denotes the matrix of B-spline basis functions defined by knot vector $\Delta_i$ for $i=1,2$, and $A$ is the transition matrix. The Oslo and knot insertion algorithms compute $A$ iteratively to represent $S_d^\mu(\Delta_2)$ basis functions using $S_d^\mu(\Delta_1)$ B-splines. In contrast, our method constructs $S_d^\mu(\Delta_2)$ basis functions by linearly combining $S_d^\mu(\Delta_1)$ B-splines, with coefficients satisfying the EEE condition (Equation~\eqref{eq one-dimension EEE}). Unlike these algorithms, which directly select B-splines as basis functions, our approach focuses on coefficient computation to enforce smoothness constraints.

The novelty of our method lies in addressing the smoothness transition at knots in $\Delta_1 \setminus \Delta_2$. Prior methods, including the Oslo and knot insertion algorithms, emphasize basis representation without explicitly handling this smoothness change. Our EEE conditions capture this transition, enabling extension to two-dimensional cases.

Table~\ref{tab:oslo_vs_new} presents the differences between the Oslo algorithm and our method in the one-dimensional case. In fact, by selecting the columns of the transition matrix \( A \) as a set of particular solutions to the Extended Edge Elimination (EEE) condition, which inherently satisfy the condition, our method constructs a set of normalized B-splines in \( S_d^{\mu}(\Delta_2) \). Consequently, our algorithm can be regarded as a generalization of the Oslo algorithm, representing a specific set of solutions to the EEE condition.

\begin{table}[ht]
\centering
\resizebox{0.9\textwidth}{!}{
\begin{tabular}{lcc}
\toprule
Aspect & Oslo Algorithm & Proposed Method \\
\midrule
Basis Representation & B-splines of $S_d^\mu(\Delta_2)$ & Linear combination of $S_d^\mu(\Delta_1)$ B-splines \\
Smoothness Handling & Adjusts knot multiplicity & Enforces $C^d$ via EEE conditions at $\Delta_1 \setminus \Delta_2$ \\
Transition Mechanism & Iteratively computes matrix $A$ & Solves linear system $M \boldsymbol{c} = 0$ \\
Novelty & Standard knot insertion & Addresses smoothness at $\Delta_1 \setminus \Delta_2$ \\
\bottomrule
\end{tabular}
}
\caption{Comparison of Oslo Algorithm and Proposed Method for $S_d^\mu(\Delta_2)$}
\label{tab:oslo_vs_new}
\end{table}

\subsubsection{A Comparative Analysis of Basis Construction}
We classify spline space basis function constructions into three categories: explicit, semi-implicit, and implicit. Using the spline space over a T-mesh as an example, we compare their advantages and disadvantages. Additionally, we briefly contrast spline basis construction with finite element basis construction to highlight the lower degrees of freedom in spline methods.

\begin{definition}\label{def E,SE,I}
The construction method $\mathcal{M}$ for spline space $S_d^\mu(\Delta)$ is categorized into three types:
\begin{itemize}
    \item \textbf{Explicit Construction}: Basis functions of $S_d^\mu(\Delta)$ are defined explicitly without solving a linear system for parameter selection.
    \item \textbf{Semi-Implicit Construction}: Basis functions of $S_d^\mu(\Delta)$ are defined explicitly but require solving a linear system for parameter selection.
    \item \textbf{Implicit Construction}: Basis functions of $S_d^\mu(\Delta)$ are determined solely by solving a linear system, without an explicit representation.
\end{itemize}
\end{definition}

Based on Definition~\ref{def E,SE,I}, the method proposed in this paper is classified as a semi-implicit construction. It derives basis functions for the spline space by solving a linear system defined by the EEE condition, using combination coefficients from a set of basis functions in the solution space. This subsection examines the advantages and limitations of explicit, semi-implicit, and implicit construction methods through examples.

\begin{example}
Consider the spline space $S_{d,d}^{\mu,\mu}(\Delta)$ defined over a T-mesh $\Delta$. We illustrate the three construction types below.

\begin{itemize}
    \item \textbf{Explicit Construction}: This approach constructs spline spaces over T-meshes using local tensor-product B-splines, tailored to specific T-mesh configurations. It generates a set of linearly independent functions, but these may not form a basis for the entire spline space. The focus is on creating a spanning set for the linear space.

    Examples include HB-splines over hierarchical meshes \cite{HB1,HB2,HB3,HB4}, THB-splines over hierarchical meshes \cite{thb}, LR-splines over specific LR-meshes \cite{patrizi2020adaptive}, and PHT-splines over hierarchical T-meshes \cite{PHT}.

    \item \textbf{semi-implicit Construction}: The proposed method introduces the first semi-implicit construction, deriving basis functions for spline spaces on arbitrary polygonal partitions. It combines an explicit construction with EEE conditions to ensure a complete basis. The explicit component must provide a basis, not merely linearly independent generators.

    Example: PT-spline over arbitrary T-meshes.

    \item \textbf{Implicit Construction}: This method defines basis functions via the null space of a linear system, enforcing smoothness constraints across cells without explicit construction. It is critical in applications like isogeometric analysis (IGA) \cite{thomas2022u}.

    Example: U-spline over arbitrary T-meshes \cite{thomas2022u}.
\end{itemize}
\end{example}

Table~\ref{tab:construction_comparison} compares explicit, semi-implicit, and implicit construction methods for spline spaces, focusing on construction approach, theoretical robustness, mesh refinement flexibility, computational complexity, and applicability to theoretical and practical contexts.

\begin{table}[ht]
\centering
\resizebox{\textwidth}{!}{
\begin{tabular}{lccc}
\toprule
Aspect & Explicit Construction & semi-implicit Construction & Implicit Construction \\
\midrule
Construction Approach & Local tensor-product B-splines & Explicit construction with EEE conditions & Null space of linear system for smoothness \\
Theoretical Robustness & Susceptible to linear dependence; incomplete basis & Resolves all theoretical issues & Robust but limited theoretical insight \\
Refinement Flexibility & Limited by mesh constraints & Enables arbitrary mesh refinement & Supports arbitrary meshes \\
Computational Complexity & Low, direct computation & High, due to EEE condition computation & Moderate, linear system solution \\
Theoretical Contribution & Hindered by basis limitations & Comprehensive, addresses all concerns & Limited, application-focused \\
Practical Applicability & Constrained by mesh restrictions & Broad, computationally demanding & Effective for engineering (e.g., IGA) \\
Future Research Focus & Mitigating linear dependence & Optimizing EEE computation & Deepening theoretical insights \\
\bottomrule
\end{tabular}
}
\caption{Comparison of Construction Methods for Spline Spaces}
\label{tab:construction_comparison}
\end{table}

Explicit construction is simple but often produces linearly dependent functions, failing to form a complete basis and requiring restrictive mesh conditions for refinement. semi-implicit construction addresses these issues through EEE conditions, supporting flexible refinement on arbitrary meshes, though computing these conditions is computationally intensive, highlighting a key area for future optimization. Implicit construction is highly effective for practical applications, such as isogeometric analysis, due to its straightforward handling of smoothness constraints, but it offers minimal theoretical advancements.

Finally, we briefly compare the differences between the construction of finite element basis functions and the construction of spline basis functions from the perspective of degrees of freedom.

Finite element construction emphasizes defining a complete degree-$d$ polynomial within each cell, where the degrees of freedom fully determine the polynomial in that cell. In contrast, spline basis construction operates globally, using degrees of freedom across multiple adjacent cells to define a piecewise polynomial, without requiring a complete polynomial per cell. Both methods produce $\mu$-order smooth piecewise polynomial basis functions on the mesh, but spline construction requires fewer degrees of freedom. Consequently, spline construction imposes less restrictive conditions on polynomial degree $d$ and smoothness $\mu$, requiring only $\mu \leq d$, unlike the stricter requirements in finite element construction.

Table~\ref{tab:fe_spline_comparison} summarizes the key differences between these construction methods.

\begin{table}[ht]
\centering
\resizebox{\textwidth}{!}{
\begin{tabular}{lcc}
\toprule
Aspect & Finite Element Construction & Spline Construction \\
\midrule
Construction Approach & Local, per-cell polynomial definition & Global, across multiple cells \\
Polynomial Requirement & Complete degree-$d$ polynomial per cell & Piecewise polynomial, no per-cell completion \\
Degrees of Freedom & High, fully determines cell polynomial & Lower, shared across adjacent cells \\
Smoothness Constraint & Strict, high degree and smoothness & Flexible, only requires $\mu \leq d$ \\
Basis Function Type & $\mu$-order smooth piecewise polynomials & $\mu$-order smooth piecewise polynomials \\
\bottomrule
\end{tabular}
}
\caption{\label{tab:fe_spline_comparison}Comparison of Finite Element and Spline Basis Construction}
\end{table}

\subsection{Examples}
In this subsection, we take the construction of spline space basis functions over Morgan-Scott partition~\cite{morgan1977dimension} as an example. We first construct a set of basis functions over Morgan-Scott partition~\cite{morgan1977dimension}, and then use this set of basis functions to explain dimensional instability as the basis function degradation caused by the EEE condition.

A Morgan-Scott partition is shown in Figure~\ref{fig:MS} denoted as $\De_{\mathrm{MS}}$, many papers have studied the dimension of $S_2^1(\De_{\mathrm{MS}})$, and the following lemma is given:

\begin{figure}
    \centering
    \begin{tikzpicture}[line cap=round,line join=round,>=triangle 45,x=1.0cm,y=1.0cm,scale=0.8]
\draw [line width=1.pt] (7.16,7.08)-- (2.94,1.3);
\draw [line width=1.pt] (2.94,1.3)-- (11.1,1.66);
\draw [line width=1.pt] (11.1,1.66)-- (7.16,7.08);
\draw [line width=1.pt] (7.16,7.08)-- (6.26,4.2);
\draw [line width=1.pt] (6.26,4.2)-- (7.86,3.82);
\draw [line width=1.pt] (7.86,3.82)-- (7.16,7.08);
\draw [line width=1.pt] (6.26,4.2)-- (2.94,1.3);
\draw [line width=1.pt] (2.94,1.3)-- (6.86,2.56);
\draw [line width=1.pt] (6.86,2.56)-- (11.1,1.66);
\draw [line width=1.pt] (11.1,1.66)-- (7.86,3.82);
\draw [line width=1.pt] (7.86,3.82)-- (6.86,2.56);
\draw [line width=1.pt] (6.86,2.56)-- (6.26,4.2);
\begin{scriptsize}
\draw [fill=black] (7.16,7.08) circle (2.5pt);
\draw[color=black] (7.2,7.45) node {$A$};
\draw [fill=black] (2.94,1.3) circle (2.5pt);
\draw[color=black] (2.8,1.67) node {$B$};
\draw [fill=black] (11.1,1.66) circle (2.5pt);
\draw[color=black] (11.24,2.03) node {$C$};
\draw [fill=black] (6.26,4.2) circle (2.5pt);
\draw[color=black] (6.2,4.57) node {$D$};
\draw [fill=black] (7.86,3.82) circle (2.5pt);
\draw[color=black] (8.,4.19) node {$E$};
\draw [fill=black] (6.86,2.56) circle (2.5pt);
\draw[color=black] (6.9,2.93) node {$F$};
\end{scriptsize}
\end{tikzpicture}
    \caption{\label{fig:MS}Morgan-Scott Partition}
\end{figure}
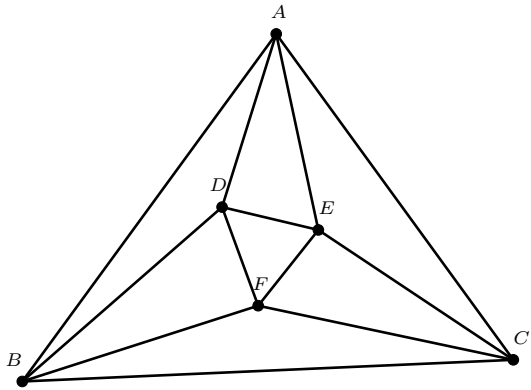

\begin{lemma}\label{lem dimension MS}~\cite{morgan1977dimension,alfeld1986dimension,gmelig1985dimension,wang2013multivariate}
\begin{equation*}
\dim S_2^1(\De_{\mathrm{MS}}) =
\begin{cases}
7, & \text{when $AF,BE,CD$ intersect at one point}, \\
6, & \text{otherwise}.
\end{cases}
\end{equation*}
\end{lemma}

Following the framework in Section 3, we construct a basis for the spline space $S_2^1(\Delta_{\text{MS}})$, achieving the first successful construction of such a basis in the literature. The construction is tailored to the specific partition $\Delta_{\text{MS}}$, as the method is inherently partition-dependent. To simplify the explanation, we impose geometric restrictions on the Morgan-Scott partition:
\begin{itemize}
    \item Extensions of edges $FE$, $DF$, and $ED$ intersect only at the interior of edges $AC$, $BC$, and $AB$ respectively.
    \item Extensions of edges $AE$, $DB$, and $FC$ do not intersect at the interior of edges $FG$, $EI$, and $HD$ respectively.
\end{itemize}
These restrictions streamline the description of the basis construction process and avoid extraneous details, without altering the core results. For partitions not satisfying these conditions, minor differences in the construction process may arise, but the final conclusions remain unaffected.

\begin{example}\label{exm construction for MS}
We construct a basis for $S_2^1(\Delta_{\mathrm{MS}})$ using the framework outlined in Section 3 as follows:
\begin{itemize}
    \item Select the basis construction method $\mathcal{M}$ for spline spaces over quasi-cross-cut partitions. Extend the Morgan-Scott partition $\Delta_{\mathrm{MS}}$ to a quasi-cross-cut partition $\mathrm{ext}_3(\Delta_{\mathrm{MS}})$, as shown in the left subfigure of Figure~\ref{fig: Extended MS}.
    \item Verify that $(\mathrm{ext}_3(\Delta_{\mathrm{MS}}), \mathcal{M})$ is compatible, then construct basis functions for $S_2^1(\mathrm{ext}_3(\Delta_{\mathrm{MS}}))$ using the method in~\cite{wang2013multivariate}:
    \begin{itemize}
        \item[1.] Define a flow direction $\overrightarrow{F}$, indicated by the red dashed line in the right subfigure of Figure~\ref{fig: Extended MS}. For an edge $l_i$ we denote
        by $U(l_i^{+})$ the union of the closed cells that the flow $\overrightarrow{F}$ will reach after it crosses $l_i$ and we denote by $U(l_i^{-})$ the union of the closed cells that the flow $\overrightarrow{F}$ has reached before it arrives at $l_i$. We call $U(l_i^{+})\setminus U(l_i^{-})$ be the frontier of $l_i$ denoted by $f_r(l_i)$, follows \cite{wang2013multivariate}. Basis functions over $\mathrm{ext}_3(\Delta_{\mathrm{MS}})$ are constructed starting from the source cell $\triangle_{\text{BFG}}$ along this flow direction.
        
        \item[2.] Compute $k_2^1(\cdot)$ from equation~\eqref{kdmu} for interior vertices $D$, $E$, and $F$ to determine their degrees of freedom. With 4 rays and 0 cross-cuts, we obtain:
        \[
        k_2^1(4) = \sum_{j=1}^{2-1} \left( 4 \times (2 - 1 - j + 1) - (2 - j + 2) \right)_+ = 1.
        \]
        Each interior vertex contributes one degree of freedom. The edge cofactors around each vertex are expressed as:
\begin{itemize}
    \item For vertex $D$:
    \[
    (q_{\text{DA}}, q_{\text{DB}}, q_{\text{DF}}, q_{\text{DE}} + q_{\text{DH}}) = k_3 (p_{\text{DA}}, p_{\text{DB}}, p_{\text{DF}}, p_3).
    \]
    \item For vertex $E$:
    \[
    (q_{\text{EA}}, q_{\text{EC}}, q_{\text{ED}}, q_{\text{EF}} + q_{\text{EI}}) = k_2 (p_{\text{EA}}, p_{\text{EC}}, p_{\text{ED}}, p_2).
    \]
    \item For vertex $F$:
    \[
    (q_{\text{FB}}, q_{\text{FC}}, q_{\text{FE}}, q_{\text{FD}} + q_{\text{FG}}) = k_1 (p_{\text{FB}}, p_{\text{FC}}, p_{\text{DF}}, p_1).
    \]
\end{itemize}
Here, $q_{l_i}$ denotes the edge cofactor for edge $l_i$, $p_{l_i} \in \mathbb{R}$ is a non-zero constant for edge $l_i$, and $p_1, p_2, p_3$ are non-zero constants. The dimension of the spline space is:
\[
\dim S_2^1(\mathrm{ext}_3(\Delta_{\text{MS}})) = \binom{4}{2} + 0 \cdot \binom{2}{2} + 3 = 9.
\]

        \item[3.] Truncated power functions are defined along the flow direction $\overrightarrow{F}$ as:
\[
[L_i(x,y)]_+^{\mu+1} =
\begin{cases}
[L_i(x,y)]^{\mu+1}, & \text{if } (x,y) \in f_r(l_i), \\
0, & \text{otherwise},
\end{cases}
\]
where $L_i(x,y) = 0$ is the equation of edge $l_i$. Every function $s(x,y) \in S_2^1(\mathrm{ext}_3(\Delta))$ can be expressed using these truncated power functions.
        \[
        s(x,y) = s_0(x,y) + \sum_{\overrightarrow{F}} q_{l_i} [L_i(x,y)]_+^{2},
        \]
        where $s_0(x,y) \in \Pi_2$ and $q_{l_i} \in \mathbb{R}$ are edge cofactors for edge $l_i$. For example, in $\triangle_{\text{CEI}}$, we have:
        \[
        s(x,y)|_{\triangle_{\text{CEI}}} = s_0(x,y) + \left(k_1p_1-k_3p_{DF}\right) [L_{\text{FG}}]_+^2 + k_1 p_{\text{FC}} [L_{\text{FC}}]_+^2 + k_2 p_{\text{EC}} [L_{\text{EC}}]_+^2,
        \]
        where $p_{\text{FG}}, p_{\text{CF}}, p_{\text{EC}} \in \mathbb{R}$ are non-zero constants for edges FG, CF, and EC, and $k_1, k_2, k_3$ are the degrees of freedom for vertices $F$, $E$, and $D$, respectively. Generally:
        \[
        s(x,y) = s_0(x,y) + k_1 s_1(x,y) + k_2 s_2(x,y) + k_3 s_3(x,y),
        \]
        where:
        \begin{align*}
        s_0(x,y) &\in \Pi_2, \\
        s_1(x,y) &= p_1[L_{FG}]_{+}^2+p_{\text{FC}} [L_{\text{FC}}]_+^2 - p_{\text{FE}} [L_{\text{EI}}]_+^2, \\
        s_2(x,y) &= p_2[L_{\text{EI}}]_+^2+p_{\text{EC}} [L_{\text{EC}}]_+^2+ p_{\text{EA}} [L_{\text{EA}}]_+^2 + p_{\text{ED}} [L_{\text{ED}}]_+^2-p_{\text{ED}} [L_{\text{DH}}]_+^2, \\
        s_3(x,y) &= p_3[L_{\text{DH}}]_+^2 +p_{\text{DB}} [L_{\text{DB}}]_+^2 + p_{\text{DF}} [L_{\text{DF}}]_+^2-p_{\text{DF}} [L_{\text{FG}}]_+^2.
        \end{align*}
        with $p_{l_i} \in \mathbb{R}$ as non-zero constants for edge $l_i$. Thus, the basis for $S_2^1(\mathrm{ext}_3(\Delta))$ is:
        \[
        \{ x^a y^b, s_1(x,y), s_2(x,y), s_3(x,y) : 0 \leq a + b \leq 2 \}.
        \]
    \end{itemize}
    \item Apply the EEE conditions to eliminate edges HD, FG, and EI. Let $\{ N_i \}_{i=1}^9$ denote the nine basis functions for $S_2^1(\mathrm{ext}_3(\Delta))$. The EEE conditions are:
    \[
    \begin{cases}
    \sum_{i=1}^9 c_i \cdot \frac{\partial^2}{\partial \boldsymbol{n}_{\text{GF}}^2} \left( N_i|_{\triangle_{\text{BFG}}} - N_i|_{\triangle_{\text{CFG}}} \right) \Big|_{\text{GF}} = 0, & \text{EEE condition for GF}, \\
    \sum_{i=1}^9 c_i \cdot \frac{\partial^2}{\partial \boldsymbol{n}_{\text{EI}}^2} \left( N_i|_{\triangle_{\text{CEI}}} - N_i|_{\triangle_{\text{AEI}}} \right) \Big|_{\text{EI}} = 0, & \text{EEE condition for EI}, \\
    \sum_{i=1}^9 c_i \cdot \frac{\partial^2}{\partial \boldsymbol{n}_{\text{DH}}^2} \left( N_i|_{\triangle_{\text{BDH}}} - N_i|_{\triangle_{\text{ADH}}} \right) \Big|_{\text{DH}} = 0, & \text{EEE condition for DH}.
    \end{cases}
    \]
    These form the linear system $M \boldsymbol{c} = 0$. Let $\{ (\tilde{c}_{1,j}, \tilde{c}_{2,j}, \ldots, \tilde{c}_{9,j}) \}_{j=1}^h$ be a basis of solutions. The basis for $S_2^1(\Delta_{\mathrm{MS}})$ is:
    \[
    \left\{ \sum_{i=1}^9 \tilde{c}_{i,j} N_i \right\}_{j=1}^h.
    \]
\end{itemize}
\end{example}

\begin{figure}
    \centering
    \begin{tikzpicture}[line cap=round,line join=round,>=triangle 45,x=1.0cm,y=1.0cm,scale=0.7]
\draw [line width=1.pt] (7.16,7.08)-- (2.94,1.3);
\draw [line width=1.pt] (2.94,1.3)-- (11.1,1.66);
\draw [line width=1.pt] (11.1,1.66)-- (7.16,7.08);
\draw [line width=1.pt] (7.16,7.08)-- (6.26,4.2);
\draw [line width=1.pt] (6.26,4.2)-- (7.86,3.82);
\draw [line width=1.pt] (7.86,3.82)-- (7.16,7.08);
\draw [line width=1.pt] (6.26,4.2)-- (2.94,1.3);
\draw [line width=1.pt] (2.94,1.3)-- (6.86,2.56);
\draw [line width=1.pt] (6.86,2.56)-- (11.1,1.66);
\draw [line width=1.pt] (11.1,1.66)-- (7.86,3.82);
\draw [line width=1.pt] (7.86,3.82)-- (6.86,2.56);
\draw [line width=1.pt] (6.86,2.56)-- (6.26,4.2);
\draw [line width=1.pt,dotted,color=red] (6.26,4.2)-- (5.23503041026208,4.443430277562753);
\draw [line width=1.pt,dotted,color=red] (6.86,2.56)-- (7.251387222026121,1.4902082597952702);
\draw [line width=1.pt,dotted,color=red] (7.86,3.82)-- (8.731538076345288,4.9181379761950605);
\draw [line width=1.pt] (19.19656,7.08)-- (14.97656,1.3);
\draw [line width=1.pt] (14.97656,1.3)-- (23.13656,1.66);
\draw [line width=1.pt] (23.13656,1.66)-- (19.19656,7.08);
\draw [line width=1.pt] (19.19656,7.08)-- (18.29656,4.2);
\draw [line width=1.pt] (18.29656,4.2)-- (19.89656,3.82);
\draw [line width=1.pt] (19.89656,3.82)-- (19.19656,7.08);
\draw [line width=1.pt] (18.29656,4.2)-- (14.97656,1.3);
\draw [line width=1.pt] (14.97656,1.3)-- (18.89656,2.56);
\draw [line width=1.pt] (18.89656,2.56)-- (23.13656,1.66);
\draw [line width=1.pt] (23.13656,1.66)-- (19.89656,3.82);
\draw [line width=1.pt] (19.89656,3.82)-- (18.89656,2.56);
\draw [line width=1.pt] (18.89656,2.56)-- (18.29656,4.2);
\draw [line width=1.pt,dotted,color=red] (18.29656,4.2)-- (17.271590410262085,4.443430277562753);
\draw [line width=1.pt,dotted,color=red] (18.89656,2.56)-- (19.28794722202612,1.4902082597952702);
\draw [line width=1.pt,dotted,color=red] (19.89656,3.82)-- (20.76809807634529,4.9181379761950605);
\draw [->,line width=1.pt,dotted,color=red] (18.282245379734327,1.8261798274840548) -- (20.390919095844797,1.9080700688864027);
\draw [->,line width=1.pt,dotted,color=red] (20.390919095844797,1.9080700688864027) -- (20.206666052689513,2.890752965714578);
\draw [->,line width=1.pt,dotted,color=red] (20.206666052689513,2.890752965714578) -- (21.250766630569455,3.6072925779851226);
\draw [->,line width=1.pt,dotted,color=red] (21.250766630569455,3.6072925779851226) -- (20.063358130235404,5.3269876474344295);
\draw [->,line width=1.pt,dotted,color=red] (20.063358130235404,5.3269876474344295) -- (19.060202673056637,4.835646199020342);
\draw [->,line width=1.pt,dotted,color=red] (19.060202673056637,4.835646199020342) -- (19.03973011270605,3.361621853778079);
\draw [->,line width=1.pt,dotted,color=red] (19.03973011270605,3.361621853778079) -- (17.360980163957908,2.6655548018581214);
\draw [->,line width=1.pt,dotted,color=red] (17.360980163957908,2.6655548018581214) -- (17.197199681153215,3.8120181814909926);
\draw [->,line width=1.pt,dotted,color=red] (17.197199681153215,3.8120181814909926) -- (18.220827698682566,5.1017894835779725);
\begin{scriptsize}
\draw [fill=black] (7.16,7.08) circle (2pt);
\draw[color=black] (7.308953031819653,7.46637020407077) node {$A$};
\draw [fill=black] (2.94,1.3) circle (2pt);
\draw[color=black] (2.8,1.6726356248546523) node {$B$};
\draw [fill=black] (11.1,1.66) circle (2.5pt);
\draw[color=black] (11.239684619132372,2.0411417111652184) node {$C$};
\draw [fill=black] (6.26,4.2) circle (2.5pt);
\draw[color=black] (6.5,4.579739194638004) node {$D$};
\draw [fill=black] (7.86,3.82) circle (2.5pt);
\draw[color=black] (7.5,4.2) node {$E$};
\draw [fill=black] (6.86,2.56) circle (2.5pt);
\draw[color=black] (6.9,2.9) node {$F$};
\draw [fill=black] (5.23503041026208,4.443430277562753) circle (2.5pt);
\draw[color=black] (5.25,4.8254099188450486) node {$H$};
\draw [fill=black] (7.251387222026121,1.4902082597952702) circle (2.5pt);
\draw[color=black] (7.390843273222002,1.8) node {$G$};
\draw [fill=black] (8.731538076345288,4.9181379761950605) circle (2.5pt);
\draw[color=black] (8.864867618464272,5.296278806908549) node {$I$};
\draw [fill=black] (19.19656,7.08) circle (2.5pt);
\draw[color=black] (19.408236199016617,7.46637020407077) node {$A$};
\draw [fill=black] (14.97656,1.3) circle (2.5pt);
\draw[color=black] (14.75,1.6726356248546523) node {$B$};
\draw [fill=black] (14.97656,1.3) circle (2.5pt);
\draw[color=black] (23.338967786329338,2.0411417111652184) node {$C$};
\draw [fill=black] (18.29656,4.2) circle (2.5pt);
\draw[color=black] (18.6,4.5) node {$D$};
\draw [fill=black] (19.89656,3.82) circle (2.5pt);
\draw[color=black] (20.104303250936578,4.1) node {$E$};
\draw [fill=black] (18.89656,2.56) circle (2.5pt);
\draw[color=black] (19,2.95) node {$F$};
\draw [fill=black] (17.271590410262085,4.443430277562753) circle (2.5pt);
\draw[color=black] (17,4.8254099188450486) node {$H$};
\draw [fill=black] (19.28794722202612,1.4902082597952702) circle (2.5pt);
\draw[color=black] (19.35,1.2) node {$G$};
\draw [fill=black] (20.76809807634529,4.9181379761950605) circle (2.5pt);
\draw[color=black] (20.964150785661236,5.296278806908549) node {$I$};
\end{scriptsize}
\end{tikzpicture}
    \caption{\label{fig: Extended MS}Extended Morgan-Scott Partition and Flow Direction}
\end{figure}
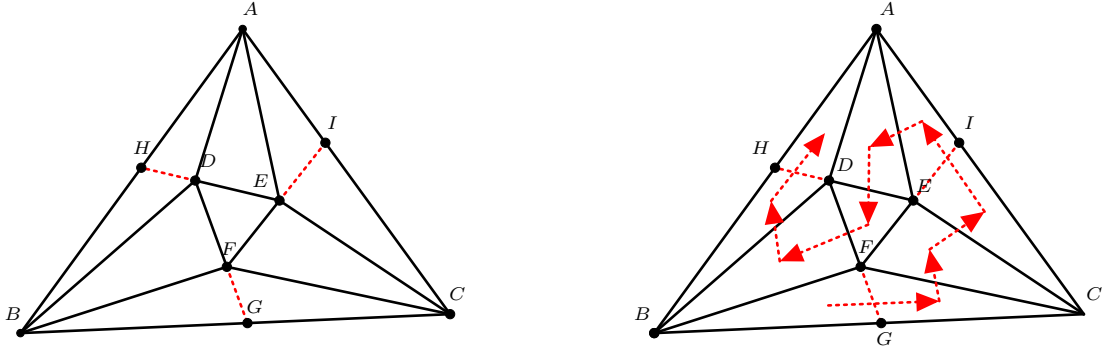

Finally, we determine the dimension $h$ of the spline space $S_2^1(\Delta_{\text{MS}})$ and analyze the dimensional instability of $S_2^1(\Delta_{\text{MS}})$ from the perspective of basis function construction.

\begin{example}
We analyze the dimensional stability of the spline space $S_2^1(\Delta_{\text{MS}})$ from Example~\ref{exm construction for MS} as follows:
\begin{itemize}
    \item For any smooth function $f(x,y)$, the second-order directional derivative is given by
    \[
    \frac{\partial^2}{\partial \boldsymbol{n}^2} f(x,y) = \boldsymbol{n} H(f) \boldsymbol{n}^{\top},
    \]
    where $H(f)$ is the Hessian matrix of $f(x,y)$. Equivalently, for $\boldsymbol{n} = (a,b)$,
    \begin{equation}\label{eq Hessian Matrix}
        \frac{\partial^2}{\partial \boldsymbol{n}^2} f(x,y) = a^2 \frac{\partial^2 f}{\partial x^2} + 2ab \frac{\partial^2 f}{\partial x \partial y} + b^2 \frac{\partial^2 f}{\partial y^2}.
    \end{equation}
    For a line $l_i$ with equation $L_i = 0$, $L_i \in \Pi_1$, the normal vector is
    \[
    \boldsymbol{n}_i = \left( \frac{\partial L_i}{\partial x}, \frac{\partial L_i}{\partial y} \right).
    \]
    For a directional vector $\boldsymbol{n}$, we have
    \begin{equation}\label{eq second order smooth}
        \frac{\partial^2}{\partial \boldsymbol{n}^2} [L_i]^2 = 2 (\boldsymbol{n}_i \boldsymbol{n}^{\top})^2.
    \end{equation}

    \item To simplify the EEE conditions, consider the condition for edge $GF$:
    \[
    \sum_{i=1}^9 c_i \cdot \frac{\partial^2}{\partial \boldsymbol{n}_{\text{GF}}^2} \left( N_i|_{\triangle_{\text{BFG}}} - N_i|_{\triangle_{\text{CFG}}} \right) \Big|_{\text{GF}} = 0.
    \]
    Let $N_1 = 1$, $N_2 = x$, $N_3 = y$, $N_4 = x^2$, $N_5 = xy$, $N_6 = y^2$, $N_7 = s_1(x,y)$, $N_8 = s_2(x,y)$, $N_9 = s_3(x,y)$. Since $N_i$, $i=1,\ldots,6$, are $C^{\infty}$ over $\mathrm{ext}_3(\Delta_{\text{MS}})$, we have 
    
    $$\frac{\partial^2}{\partial \boldsymbol{n}_{\text{GF}}^2} \left( N_i|_{\triangle_{\text{BFG}}} - N_i|_{\triangle_{\text{CFG}}} \right) \Big|_{\text{GF}} = 0; i=1,\ldots,6.$$ 
    
    Thus, only basis functions $N_7$, $N_8$, $N_9$ are considered:
    \begin{itemize}
        \item[1.] For $N_7 = p_1 [L_{FG}]_+^2 + p_{\text{FC}} [L_{\text{FC}}]_+^2 - p_{\text{FE}} [L_{\text{EI}}]_+^2$, from Figure~\ref{fig:cross-cut MS}, $[L_{\text{EI}}]_+^2 = [L_{II'}]_+^2$ and $[L_{\text{FC}}]_+^2 = [L_{CC''}]_+^2$ are $C^{\infty}$ in $GF$ except at $F$. Since $N_i \in C^1(\Delta)$, $i=1,\ldots,9$, then $\sum c_i N_i$ is $C^1$ at $F$, then edge $FG$ (except $F$) requires $C^2$ continuity. 
        
        Thus,
        
       $$ \frac{\partial^2}{\partial \boldsymbol{n}_{\text{GF}}^2} \left( N_7|_{\triangle_{\text{BFG}}} - N_7|_{\triangle_{\text{CFG}}} \right) \Big|_{\text{GF}}  =  \frac{\partial^2}{\partial \boldsymbol{n}_{\text{GF}}^2} \left( p_1 [L_{FG}]_+^2|_{\triangle_{\text{BFG}}} - p_1 [L_{FG}]_+^2|_{\triangle_{\text{CFG}}} \right) \Big|_{\text{GF}}$$
       Further, we have
       $$ \frac{\partial^2}{\partial \boldsymbol{n}_{\text{GF}}^2} \left( N_7|_{\triangle_{\text{BFG}}} - N_7|_{\triangle_{\text{CFG}}} \right) \Big|_{\text{GF}}= -p_1 \frac{\partial^2}{\partial \boldsymbol{n}_{\text{GF}}^2} [L_{FG}]^2.$$

        \item[2.] For $N_8 = p_2 [L_{\text{EI}}]_+^2 + p_{\text{EC}} [L_{\text{EC}}]_+^2 + p_{\text{EA}} [L_{\text{EA}}]_+^2 + p_{\text{ED}} [L_{\text{ED}}]_+^2 - p_{\text{ED}} [L_{\text{DH}}]_+^2$, from Figure~\ref{fig:cross-cut MS}, $[L_{\text{EC}}]_+^2 = [L_{CC'}]_+^2$, $[L_{\text{DH}}]_+^2 = [L_{H'H}]_+^2$, $[L_{\text{ED}}]_+^2 = [L_{H'H}]_+^2$, $[L_{\text{EI}}]_+^2 = [L_{II'}]_+^2$ are $C^{\infty}$ in $FG$ except at $F$, and $[L_{\text{EA}}]_+^2 = [L_{AA'}]_+^2$ is $C^{\infty}$ in $FG$ since the extension of $AE$ does not intersect $FG$ internally. Thus,
        \[
        \frac{\partial^2}{\partial \boldsymbol{n}_{\text{GF}}^2} \left( N_8|_{\triangle_{\text{BFG}}} - N_8|_{\triangle_{\text{CFG}}} \right) \Big|_{\text{GF}} = 0.
        \]

        \item[3.] For $N_9 = p_3 [L_{\text{DH}}]_+^2 + p_{\text{DB}} [L_{\text{DB}}]_+^2 + p_{\text{DF}} [L_{\text{DF}}]_+^2 - p_{\text{DF}} [L_{\text{FG}}]_+^2$, from Figure~\ref{fig:cross-cut MS}, $[L_{\text{DB}}]_+^2 = [L_{BB''}]_+^2$, $[L_{\text{DH}}]_+^2 = [L_{HH'}]_+^2$ are $C^{\infty}$ in $GF$, while $[L_{\text{GF}}]_+^2 = [L_{GG'}]_+^2$, $[L_{\text{DF}}]_+^2 = [L_{GG'}]_+^2$ are discontinuous along $GF$. Thus,
        \[
        \frac{\partial^2}{\partial \boldsymbol{n}_{\text{GF}}^2} \left( N_9|_{\triangle_{\text{BFG}}} - N_9|_{\triangle_{\text{CFG}}} \right) \Big|_{\text{GF}} = 2 p_{\text{DF}} \cdot \frac{\partial^2}{\partial \boldsymbol{n}_{\text{GF}}^2} [L_{\text{DF}}]^2.
        \]
    \end{itemize}
    Using Equation~\eqref{eq Hessian Matrix}, the EEE condition for $GF$ simplifies to
    \[
    \left( -p_1 c_7 + 2 p_{\text{DF}} c_9 \right) \frac{\partial^2}{\partial \boldsymbol{n}_{\text{FG}}^2} [L_{FG}]^2 = 0.
    \]
    By Equation~\eqref{eq second order smooth}, this reduces to
    \[
    -p_1 c_7 + 2 p_{\text{DF}} c_9 = 0.
    \]
    Applying the same method to all EEE conditions yields
    \begin{equation}\label{eq global EEE conditions MS}
        \begin{pmatrix}
            -p_1 & 0 & 2 p_{\text{DF}} \\
            p_{\text{FE}} & -p_2 & 0 \\
            0 & 2 p_{\text{ED}} & -p_3
        \end{pmatrix}
        \begin{pmatrix}
            c_7 \\ c_8 \\ c_9
        \end{pmatrix}
        =
        \begin{pmatrix}
            0 \\ 0 \\ 0
        \end{pmatrix}.
    \end{equation}
    \item Basis construction:
    
    For the matrix
\[
M = \begin{pmatrix}
    -p_1 & 0 & 2 p_{\text{DF}} \\
    p_{\text{FE}} & -p_2 & 0 \\
    0 & 2 p_{\text{ED}} & -p_3
\end{pmatrix},
\]
the rank is
\[
\mathrm{rank}(M) =
\begin{cases}
    2 & \text{if } p_1 p_2 p_3 = 4 p_{\text{ED}} p_{\text{DF}} p_{\text{FE}}, \\
    3 & \text{otherwise}.
\end{cases}
\]

Thus, the dimension of $S_2^1(\Delta_{\text{MS}})$ is
\[
h = \dim S_2^1(\Delta_{\text{MS}}) = 9 - \mathrm{rank}(M) =
\begin{cases}
    7 & \text{if } p_1 p_2 p_3 = 4 p_{\text{ED}} p_{\text{DF}} p_{\text{FE}}, \\
    6 & \text{otherwise}.
\end{cases}
\]
This aligns with Lemma~\ref{lem dimension MS}, where equality conditions are expressed geometrically, while here they are algebraic.

The basis functions for the two cases are:
\begin{itemize}
    \item[1.] If $p_1 p_2 p_3 \neq 4 p_{\text{ED}} p_{\text{DF}} p_{\text{FE}}$, then $\mathrm{rank}(M) = 3$, so $(c_7, c_8, c_9) = (0, 0, 0)$. The basis is $\{ N_i \}_{i=1}^6$, equivalent to the basis for $\Pi_2$. All interior edges in $\Delta_{\text{MS}}$ are degenerate.
    
    \item[2.] 
If $p_1 p_2 p_3 = 4 p_{\text{ED}} p_{\text{DF}} p_{\text{FE}}$, then $\mathrm{rank}(M) = 2$, allowing a solution 
$$(c_7, c_8, c_9) = (2p_2p_{DF}, 2p_{DF}p_{FE}, p_1p_2).$$ 
The basis for $S_2^1(\Delta)$ is 
$$\{ N_i \}_{i=1}^6\cup\{2p_2p_{DF}\cdot N_7+2p_{DF}p_{FE}\cdot N_8+p_1p_2\cdot N_9\}.$$ 
The basis function $2p_2p_{DF}\cdot N_7+2p_{DF}p_{FE}\cdot N_8+p_1p_2\cdot N_9$ satisfies the EEE conditions, indicating automatic degeneracy of the extended edges $FG$, $EI$, and $HD$.
\end{itemize}
\end{itemize}
\end{example}


This example constructs basis functions for the spline space over the Morgan-Scott partition, applicable to any such partition despite mild geometric constraints. We show that dimensional instability stems from the rank deficiency of the coefficient matrix in the Extended Edge Elimination (EEE) conditions, which is determined by edge cofactors associated with specific geometric configurations of the partition. This rank deficiency causes basis functions over the extended quasi-cross-cut partition to form linear combinations that produce additional basis functions satisfying the EEE conditions, thereby increasing the total number of basis functions. The basis construction is summarized as follows:
\begin{itemize}
    \item Basis functions can be constructed for dimensional unstable partitions, with instability defined by the EEE conditions.
    \item Dimensional instability results from degenerate EEE conditions.
    \item Instability occurs in partitions with both non-degenerate and degenerate smooth edge cofactors, as exemplified by the Morgan-Scott partition.
\end{itemize}

\begin{figure}
    \centering

\begin{tikzpicture}[line cap=round,line join=round,>=triangle 45,x=1.0cm,y=1.0cm,scale=0.85]
\draw [line width=1.pt] (7.595568876727873,8.868740588085977)-- (2.94,1.3);
\draw [line width=1.pt] (2.94,1.3)-- (11.1,1.66);
\draw [line width=1.pt] (11.1,1.66)-- (7.595568876727873,8.868740588085977);
\draw [line width=1.pt] (7.595568876727873,8.868740588085977)-- (6.34674269534206,4.344304750606254);
\draw [line width=1.pt] (6.34674269534206,4.344304750606254)-- (8.127855445843137,4.2828870695544925);
\draw [line width=1.pt] (8.127855445843137,4.2828870695544925)-- (7.595568876727873,8.868740588085977);
\draw [line width=1.pt] (6.34674269534206,4.344304750606254)-- (2.94,1.3);
\draw [line width=1.pt] (2.94,1.3)-- (7.206590230066718,2.8702804053639905);
\draw [line width=1.pt] (7.206590230066718,2.8702804053639905)-- (11.1,1.66);
\draw [line width=1.pt] (11.1,1.66)-- (8.127855445843137,4.2828870695544925);
\draw [line width=1.pt] (8.127855445843137,4.2828870695544925)-- (7.206590230066718,2.8702804053639905);
\draw [line width=1.pt] (7.206590230066718,2.8702804053639905)-- (6.34674269534206,4.344304750606254);
\draw [line width=1.pt,dotted,color=red] (6.34674269534206,4.344304750606254)-- (4.844431341600091,4.396108590390461);
\draw [line width=1.pt,dotted,color=red] (7.206590230066718,2.8702804053639905)-- (7.992558065337759,1.5229069734707834);
\draw [line width=1.pt,dotted,color=red] (8.127855445843137,4.2828870695544925)-- (9.10015682426621,5.773749183136528);
\draw [line width=1.pt,dash pattern=on 2pt off 2pt,color=blue] (8.127855445843137,4.2828870695544925)-- (9.85385259871687,4.2233699263519515);
\draw [line width=1.pt,dash pattern=on 2pt off 2pt,color=blue] (7.206590230066718,2.8702804053639905)-- (6.2785524070790455,1.4472890767828988);
\draw [line width=1.pt,dash pattern=on 2pt off 2pt,color=blue] (6.34674269534206,4.344304750606254)-- (5.599990702664442,5.624451023767874);
\draw [line width=1.pt,dash pattern=on 1pt off 2pt on 5pt off 4pt,color=green] (6.34674269534206,4.344304750606254)-- (5.538097212535242,1.4146219358471428);
\draw [line width=1.pt,dash pattern=on 1pt off 2pt on 5pt off 4pt,color=green] (8.127855445843137,4.2828870695544925)-- (8.445888862838341,1.542906861595809);
\draw [line width=1.pt,dash pattern=on 1pt off 2pt on 5pt off 4pt,color=green] (8.127855445843137,4.2828870695544925)-- (5.954524033117262,6.200829739039624);
\draw [line width=1.pt,dash pattern=on 1pt off 2pt on 5pt off 4pt,color=green] (7.206590230066718,2.8702804053639905)-- (4.43570182540568,3.7316210141818);
\draw [line width=1.pt,dash pattern=on 1pt off 2pt on 5pt off 4pt,color=green] (6.34674269534206,4.344304750606254)-- (8.750728995804632,6.4925347665968625);
\draw [line width=1.pt,dash pattern=on 1pt off 2pt on 5pt off 4pt,color=green] (7.206590230066718,2.8702804053639905)-- (10.010048988069295,3.9020683478561096);
\begin{scriptsize}
\draw [fill=black] (7.595568876727873,8.868740588085977) circle (2.5pt);
\draw[color=black] (7.738876799181982,9.247482954571836) node {$A$};
\draw [fill=black] (2.94,1.3) circle (2.5pt);
\draw[color=black] (2.9,1.6726356248546521) node {$B$};
\draw [fill=black] (11.1,1.66) circle (2.5pt);
\draw[color=black] (11.239684619132372,2.041141711165218) node {$C$};
\draw [fill=black] (6.34674269534206,4.344304750606254) circle (2.5pt);
\draw[color=black] (6.3,4.7) node {$D$};
\draw [fill=black] (8.127855445843137,4.2828870695544925) circle (2.5pt);
\draw[color=black] (8.271163368297247,4.7) node {$E$};
\draw [fill=black] (7.206590230066718,2.8702804053639905) circle (2.5pt);
\draw[color=black] (7.25,3.24902277184985) node {$F$};
\draw [fill=black] (4.844431341600091,4.396108590390461) circle (2.5pt);
\draw[color=black] (4.8,4.7) node {$H$};
\draw [fill=black] (7.992558065337759,1.5229069734707834) circle (2.5pt);
\draw[color=black] (8.1,1.8) node {$G$};
\draw [fill=black] (9.10015682426621,5.773749183136528) circle (2.5pt);
\draw[color=black] (9.23337370477484,6.) node {$I$};
\draw [fill=black] (9.85385259871687,4.2233699263519515) circle (2.0pt);
\draw[color=black] (9.99085843774656,4.559266634287417) node {$H'$};
\draw [fill=black] (6.2785524070790455,1.4472890767828988) circle (2.0pt);
\draw[color=black] (6.3,1.7) node {$I'$};
\draw [fill=black] (5.599990702664442,5.624451023767874) circle (2.0pt);
\draw[color=black] (5.5,5.9) node {$G'$};
\draw [fill=black] (5.538097212535242,1.4146219358471428) circle (2.0pt);
\draw[color=black] (5.5,1.7) node {$A'$};
\draw [fill=black] (8.445888862838341,1.542906861595809) circle (2.0pt);
\draw[color=black] (8.65,1.9) node {$A''$};
\draw [fill=black] (10.010048988069295,3.9020683478561096) circle (2.0pt);
\draw[color=black] (10.3,4.) node {$B'$};
\draw [fill=black] (8.750728995804632,6.4925347665968625) circle (2.0pt);
\draw[color=black] (8.936521579691327,6.88290223407904) node {$B''$};
\draw [fill=black] (5.954524033117262,6.200829739039624) circle (2.0pt);
\draw[color=black] (5.9,6.5) node {$C'$};
\draw [fill=black] (4.43570182540568,3.7316210141818) circle (2.0pt);
\draw[color=black] (4.3,4) node {$C''$};
\end{scriptsize}
\end{tikzpicture}
    \caption{\label{fig:cross-cut MS}Extension of Morgan-Scott Partition}
\end{figure}
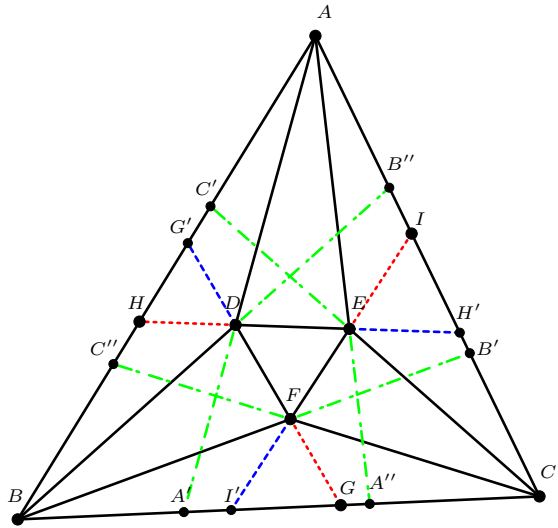

Finally, we derive a dimension formula for the spline space $S_d^{\mu}(\Delta)$ over an arbitrary partition $\Delta$ of a domain $\Omega$, utilizing the EEE conditions.

\begin{theorem}
Let $\Delta$ be a partition of $\Omega$, and let $\mathrm{ext}_s(\Delta)$ be the extended cross-cut partition of $\Delta$ with $L$ cross-cuts and $V$ interior vertices $A_1, \dots, A_V$ in $\Omega$, where $N_i$ cross-cuts intersect at $A_i$ for $i=1, \dots, V$. The dimension of the spline space $S_d^{\mu}(\Delta)$ is given by
\[
\dim S_d^{\mu}(\Delta) = \binom{d+2}{2} + L \binom{d-\mu-1}{2} + \sum_{i=1}^V k_d^{\mu}(N_i) - \mathrm{rank}(M),
\]
where $M$ is the coefficient matrix of the EEE conditions.
\end{theorem}

\begin{theorem}
Let $\Delta$ be a partition of $\Omega$, and let $\mathrm{ext}_s(\Delta)$ be the extended quasi-cross-cut partition of $\Delta$ with $L_2$ rays and $V$ interior vertices $A_1, \dots, A_V$ in $\Omega$, where $N_i$ cross-cuts and rays intersect at $A_i$ for $i=1, \dots, V$. The dimension of the spline space $S_d^{\mu}(\Delta)$ is given by
\[
\dim S_d^{\mu}(\Delta) = \binom{d+2}{2} + L_2 \binom{d-\mu-1}{2} + \sum_{i=1}^V k_d^{\mu}(N_i) - \mathrm{rank}(M),
\]
where $M$ is the coefficient matrix of the EEE conditions.
\end{theorem}


\section{Conclusion and future work}
  This study presents a framework for constructing basis functions for spline spaces over arbitrary meshes. We first address limitations of knot insertion algorithms in one-dimensional cases by analyzing properties of refined and coarse spline spaces, establishing a general framework. This framework is extended to two-dimensional polygonal partitions. We introduce the concept of Basis Construction-Suitable partitions, as defined in Definition~\ref{def CSP}, and clarify the relationship between dimensional stability and basis function construction: basis functions require a dimensional stable mesh. We define compatibility between a polygonal mesh $\De$ and a basis construction method $\mathcal{M}$, showing that an extended mesh $\mathrm{ext}_s(\De)$ (penetrating or quasi-penetrating) is compatible with the method. Using the EEE condition from \cite{zhong2025basisconstructionpolynomialspline}, we develop Algorithm~\ref{alg:basis_construction}, a general framework for two-dimensional basis construction, demonstrating that basis functions can be constructed on dimensionally unstable meshes, challenging prior assumptions.

We classify spline basis construction methods into three categories---explicit, semi-implicit, and implicit---based on their basis function representations, and provide a comparative analysis of their strengths and limitations. We also distinguish spline basis construction from finite element basis construction, noting that splines offer fewer degrees of freedom, suggesting potential as an alternative to traditional finite element methods. Applying our framework to the Morgan-Scott partition, we construct a basis and show that it exhibits degenerate smooth edge cofactors, confirming dimensional instability in both non-degenerate and degenerate meshes, as discussed in Section 4. This provides a valuable case study for further investigation of dimensional instability.

Finally, we would like to emphasize that this work is not only advances the theoretical understanding of spline spaces but also enhances adaptive refinement applications. By targeting only cells not meeting the error threshold $\epsilon$ for refinement, our framework avoids refining cells that satisfy $\epsilon$. This ensures error reduction in targeted cells without affecting those meeting the threshold, minimizing unnecessary computations.

Future research directions include:
\begin{itemize}
    \item Simplifying the EEE condition: As the EEE condition relies on polynomial derivatives, developing efficient methods to generate it is critical.
    \item Investigating dimensional instability: Algorithm 1 integrates dimensional instability into the EEE condition, offering a pathway to study instability via these conditions, as noted in Section 4.
    \item Optimizing basis functions: The semi-implicit method yields basis functions with coefficients satisfying the EEE condition. Selecting coefficients to enhance basis function properties is a key question for future exploration.
\end{itemize}

	
\bibliographystyle{plainnat}
\bibliography{ref}

\end{document}